\documentclass[a4paper,reqno,oneside]{amsart}

\usepackage[T1]{fontenc}
\usepackage[utf8]{inputenc}
\usepackage[main=english,german]{babel}
\usepackage{quiver}
\usepackage{dsfont}
\usepackage{amscd,amssymb,amsthm,amsfonts}
\usepackage{mathtools}
\usepackage{stmaryrd} 
\usepackage{graphicx}
\usepackage{svg}
\usepackage{tikz}
\usepackage{tikz-cd}
\usetikzlibrary{calc,matrix,arrows,arrows.meta,decorations.pathmorphing,positioning,braids}
\usepackage[normalem]{ulem}
\usepackage[a4paper]{geometry}
\usetikzlibrary{babel}
\usepackage{hyperref}
\hypersetup{colorlinks=true,pageanchor=false,linkcolor=blue, citecolor=brown, urlcolor=red}
\usepackage[all]{hypcap}

\allowdisplaybreaks

\tikzset{
  strand/.style={line width=1.2pt, line cap=round},
  unit/.style={circle, draw, line width=0.9pt, inner sep=1.2pt, fill=white},
  eq/.style={font=\Large},
  open dot/.style={circle,draw,fill=white,inner sep=1pt},
  closed dot/.style={circle,fill=black,inner sep=1pt},
coupon/.style={draw,fill=white,rounded corners=1pt,minimum width=6pt,minimum height=4pt,inner sep=1pt}
}

\newtheorem{theorem}{Theorem}[section]

\newtheorem{lemma}[theorem]{Lemma}

\theoremstyle{definition}
\newtheorem{definition}[theorem]{Definition}
\newtheorem{example}[theorem]{Example}

\theoremstyle{remark}
\newtheorem{remark}[theorem]{Remark}

\newcommand{\N}{\mathbb{N}}

\renewcommand{\epsilon}{\varepsilon}
\renewcommand{\theta}{\vartheta}
\renewcommand{\phi}{\varphi}
\renewcommand{\Gamma}{\varGamma}
\renewcommand{\Sigma}{\varSigma}

\newcommand{\id}{\mathrm{id}}

\makeatletter
\newcommand{\subalign}[1]{
  \vcenter{
    \Let@ \restore@math@cr \default@tag
    \baselineskip\fontdimen10 \scriptfont\tw@
    \advance\baselineskip\fontdimen12 \scriptfont\tw@
    \lineskip\thr@@\fontdimen8 \scriptfont\thr@@
    \lineskiplimit\lineskip
    \ialign{\hfil$\m@th\scriptstyle##$&$\m@th\scriptstyle{}##$\crcr
      #1\crcr
    }
  }
}
\makeatother

\def\clap#1{\hbox to 0pt{\hss#1\hss}}

\newcommand\arxiv[2]{\href{https://arXiv.org/abs/#1}{\texttt{arXiv:\allowbreak #1} #2}}
\newcommand\doi[2]{\href{https://doi.org/#1}{#2}}

\makeatletter
\def\namedlabel#1#2{\begingroup
    #2%
    \def\@currentlabel{#2}%
    \phantomsection\label{#1}\endgroup
}
\makeatother

\numberwithin{equation}{section}

\begin{document}
	\title{Dynamical Systems as Functorial Realisations of Abstract Evolution Shapes}

	\author[B.~Wang]{Bangxin Wang}
	
	\thanks{Fachbereich Mathematik, Universit\"at Hamburg, Bundesstr.\ 55,
    20146 Hamburg, Germany.\\
	\texttt{bangxin.wang@uni-hamburg.de}}

\begin{abstract}
We develop a categorical framework for closed dynamical systems in which the
abstract pattern of admissible evolutions is separated from its concrete
realisation. A closed dynamical system is formulated as a functor
$X\colon \mathcal S\to\mathcal C$ from a small category $\mathcal S$,
viewed as an abstract evolution shape, to a coefficient category
$\mathcal C$. By varying $\mathcal S$ and $\mathcal C$, this single
definition encompasses many important examples including autonomous, non-autonomous, switched, hybrid, and
stochastic systems.

Within this framework, we introduce invariant subsystems, equilibria, and orbits in functorial terms. We then formulate convergence by combining a cosieve-based intrinsic notion of eventuality on the evolution shape with neighbourhood filters of invariant subsystems.

Finally, we establish a categorical Lyapunov principle based on categorical sublevel neighbourhoods. This yields abstract stability and convergence criteria that recover the classical Lyapunov method in standard examples.
\end{abstract}
	
	\maketitle
	\setcounter{tocdepth}{2}
	{\small
		\vspace*{-3em}
		\tableofcontents
	}

    \clearpage

    \section{Introduction}

Dynamical systems provide a common language for time-dependent phenomena
across mathematics, physics, engineering, and biology; see, for example,
\cite{BS02,HSD13,St18,Kh02}. Familiar examples include continuous-time
flows, discrete-time iterated maps, switched systems,
hybrid systems, and stochastic systems. These models differ substantially in
their analytic details. Nevertheless, they share a basic structural feature: they specify admissible evolutions, and admissible evolutions can be composed.

This compositional feature suggests a categorical point of view on dynamical systems. Several categorical approaches to dynamics have been developed from different perspectives. In early categorical work, dynamical systems already appear as examples of structures associated with monoids and actions~\cite{La86}. A related approach views a dynamical system as an action of a time object on a state object internal to a category with finite products~\cite{BKPSS17}. Another line of work emphasizes openness and interfaces: systems are equipped with input-output boundaries and are composed by connecting these boundaries. This appears in operadic wiring-diagram approaches to open dynamical systems~\cite{VSL15}, as well as in cospan-based approaches to open systems, where interfaces are represented by the feet of a cospan and composition is given by gluing along shared interfaces~\cite{BC19,BCV21}. A control-theoretic input-state-output perspective treats inputs as external
control parameters acting on the internal state, while outputs are read from the resulting state. Categorical versions of this idea have been developed in~\cite{BE14,My20,My23,LM25}.

\medskip

In the present paper, we take a different but related perspective. We do not begin with a specific class of systems, such as flows or monoid actions. Nor do we begin with a fixed interface-based formalism, such as wiring diagrams or operadic composition. Instead, we isolate the abstract pattern of admissible evolutions and then study its concrete realisations in different categories of state spaces.

The abstract pattern of admissible evolutions will be called the \textit{syntax} of
the system. It is encoded by a small category $\mathcal S$, whose objects represent stages, times, modes, or configurations, and whose
morphisms represent admissible evolutions between them. A concrete dynamical
system is then a realisation of this syntax in a chosen coefficient category $\mathcal C$ of state spaces. This is the \textit{semantics} of the system. Thus, a categorical dynamical system for us is a functor
\begin{equation}
    X\colon \mathcal S\longrightarrow \mathcal C \ .
\end{equation}
The functor $X$ assigns to each stage $s\in\mathcal S$ a state space
$X(s)\in\mathcal C$, and to each admissible evolution
$\alpha\colon s\to t$ a morphism
\begin{equation}
    X(\alpha)\colon X(s)\to X(t) \ .
\end{equation}
Functoriality expresses compatibility with composition.

This syntax--semantics separation is the organising principle of the paper:
$\mathcal S$ specifies which evolutions are allowed, while $\mathcal C$
specifies the kind of state spaces and maps used to realize them. This allows many familiar classes of systems to
be treated uniformly. In Sections~\ref{sec: monoid type},~\ref{sec: poset type model},~\ref{sec: switching systems}, and~\ref{sec: stochastic models}, we explain
how this framework recovers several classical examples.

\medskip

Under this framework, we develop a categorical language for the basic notions
of dynamics. The first step is to formulate \textit{states}, \textit{subsystems}, \textit{invariance},
\textit{orbits}, and \textit{convergence} without referring to a particular analytic model. A state of the system at a stage $s\in\mathcal S$ is a morphism
\begin{equation}
    x\colon 1\to X(s)
\end{equation}
from the terminal object of $\mathcal C$. More generally, an
\textit{$A$-generalised state} is a morphism
\begin{equation}
    a\colon A\to X(s) \ .
\end{equation}
This allows one to treat points, families of points, subsets, distributions,
or other generalised configurations in the same formal language, depending
on the choice of the coefficient category $\mathcal C$.

Subobjects provide the categorical replacement for subsets of state spaces.
A subsystem of $X$ is a pair
\begin{equation}
    (U,\eta)
\end{equation}
consisting of a functor
\begin{equation}
    U\colon \mathcal S\to\mathcal C
\end{equation}
and a natural transformation
\begin{equation}
    \eta\colon U\Longrightarrow X
\end{equation}
such that each component
\begin{equation}
    \eta_s\colon U(s)\to X(s)
\end{equation}
is a monomorphism. Naturality means that admissible evolutions of $X$ carry
the subsystem at one stage into the subsystem at every later admissible
stage. A particularly important case is that of an \textit{invariant subsystem}. This is a subsystem whose underlying object does not vary with $s$: it is a natural transformation
\begin{equation}
    \iota\colon \Delta I\Longrightarrow X
\end{equation}
from the constant functor at an object $I\in\mathcal C$, with each component
\begin{equation}
    \iota_s\colon I\to X(s)
\end{equation}
monic. The naturality condition says that, for every admissible evolution
$\alpha\colon s\to t$, the diagram
\begin{equation}
\begin{tikzcd}
    I \arrow[r,"\iota_s"] \arrow[dr,"\iota_t"']
        & X(s) \arrow[d,"X(\alpha)"] \\
        & X(t)
\end{tikzcd}
\end{equation}
commutes.

In many examples, especially non-autonomous or mode-dependent systems, one chooses a \textit{state-space identification}. This means an object $M\in \mathcal C$ together with isomorphisms
\begin{equation}
    \theta_s\colon M\xrightarrow{\sim}X(s) \ ,
    \qquad s\in\mathcal S \ .
\end{equation}
Such a identification allows one to compare states living at different
stages by identifying all state spaces with the common model $M$.

With this terminology, an invariant state is an invariant subsystem with
$I=1$. Equivalently, it is a natural transformation
\begin{equation}
    \iota\colon \Delta 1\Longrightarrow X \ .
\end{equation}
If a state-space identification is fixed, an \textit{equilibrium} is a state
\begin{equation}
    x_\infty\colon 1\to M
\end{equation}
such that the associated states
\begin{equation}
    \iota_s=\theta_s\circ x_\infty\colon 1\to X(s)
\end{equation}
form an invariant state of $X$.

To formulate convergence in this language, we need to make precise two
ingredients which are often implicit in classical examples. First, one has to
specify what it means to go sufficiently far along the allowed evolutions.
Second, one has to specify what it means to be sufficiently close to the
limiting object.

The first ingredient is encoded by an \textit{eventuality filter} on the
source category $\mathcal S$. For each stage $s\in\mathcal S$, this filter
selects certain collections of morphisms with source $s$,
\begin{equation}
    \alpha\colon s\to t \ ,
\end{equation}
which are to be regarded as sufficiently far along the admissible evolutions
starting from $s$. Thus an eventuality filter plays the role of a
neighbourhood filter at infinity on the set of outgoing evolutions from each stage. The second ingredient is encoded by a neighbourhood filter of an invariant
subsystem
\begin{equation}
    \iota\colon \Delta I\Longrightarrow X \ .
\end{equation}
A neighbourhood of $I$ is given stagewise by monomorphisms
\begin{equation}
    U_t\hookrightarrow X(t)
\end{equation}
through which the component $\iota_t\colon I\to X(t)$ factors.

With these two filters in place, convergence can be expressed purely in terms of factorisation as in Definition~\ref{def: convergence}. Let
\begin{equation}
    a\colon A\to X(s)
\end{equation}
be an $A$-generalised state over $s$. We say that $a$ converges to the
invariant subsystem $I$ if, for every neighbourhood
\begin{equation}
    U=(U_t\hookrightarrow X(t))_{t\in\mathcal S}
\end{equation}
of $I$, the evolved state
\begin{equation}
    A\xrightarrow{a}X(s)\xrightarrow{X(\alpha)}X(t)
\end{equation}
eventually factors through
\begin{equation}
    U_t\hookrightarrow X(t) \ .
\end{equation}

A feature of this formulation is that it separates the formal
structure of convergence from the analytic structures used in concrete
examples. The eventuality filter belongs to the syntax: it specifies which
admissible evolutions are regarded as sufficiently far along. The
neighbourhood filter belongs to the semantics: it specifies what closeness to
the invariant subsystem means in the chosen coefficient category. In
Example~\ref{ex: convergence discrete non-auto}, we consider discrete-time
non-autonomous systems, where the eventuality filter is the usual tail filter on time and the neighbourhood filter is the ordinary topological
neighbourhood filter. In this case, Definition~\ref{def: convergence}
recovers the classical notion of convergence.

\medskip

The final part of the paper develops a categorical version of the Lyapunov
method. Classically, Lyapunov theory provides a way to prove stability and
convergence without explicitly solving the equations of motion. The basic
idea is to find a nonnegative function which vanishes on the desired
equilibria and which does not increase along trajectories;
see, for example, \cite[Ch.~4]{Kh02}. Recent work has also begun to formulate
Lyapunov theory categorically, especially in the setting of flows and related
system-theoretic frameworks~\cite{AMT25,AMM25}.

Our first observation is that the core of the classical argument is not the
numerical value of the Lyapunov function by itself. Rather, the argument uses
the sublevel regions determined by those values. If the Lyapunov value does
not increase along evolution, then these sublevel regions are preserved by
the dynamics. This suggests separating the neighbourhood-theoretic mechanism
of the Lyapunov method from the particular choice of a numerical Lyapunov
function.

We implement this separation in two steps. First, we define
\textit{Lyapunov stability} for an invariant subsystem in
Definition~\ref{def: Lyapunov stable}. This means that every prescribed
neighbourhood of the invariant subsystem contains a smaller neighbourhood
whose future evolutions remain inside the prescribed one. We then introduce
the notion of a \textit{Lyapunov basis} in
Definition~\ref{def: lyapunov basis}. This is a basis of neighbourhoods of an
invariant subsystem consisting of \textit{forward-invariant neighbourhoods}:
each admissible evolution sends the chosen neighbourhood at the source stage
into the corresponding neighbourhood at the target stage. In Lemma~\ref{prop: Lyapunov stability}, we show that the existence of a Lyapunov basis implies Lyapunov
stability. Second, we explain how Lyapunov bases arise from Lyapunov functions. For this
purpose, we introduce \textit{Lyapunov measurement data} in
Definition~\ref{def: lyapunov measurement}. Such data provide the categorical
analogue of the target $\mathbb R_{\geq 0}$ of an ordinary Lyapunov function:
a measurement object $R$, a distinguished zero value, a family of sublevel
neighbourhoods of zero, and an order relation on measurements. A
\textit{categorical Lyapunov function}, defined in
Definition~\ref{def: lyapunov function}, is then a compatible family of
measurements
\begin{equation}
    L_s\colon X(s)\to R
\end{equation}
on the stages of the system. It vanishes on the desired invariant subsystem and is
non-increasing along admissible evolutions. Its sublevel neighbourhoods are the categorical analogue of the classical sublevel sets of a real-valued
Lyapunov function. Theorem~\ref{thm: lyapunov function produces Lyapunov basis}
shows that these sublevel neighbourhoods form a Lyapunov basis. Consequently,
a categorical Lyapunov function implies Lyapunov stability by
Lemma~\ref{prop: Lyapunov stability}.

Finally, we formulate convergence criteria in terms of \textit{Lyapunov
decay}. This means that the Lyapunov measurement of an evolved state
eventually lies in every sublevel neighbourhood of zero. When the sublevel
neighbourhoods of the Lyapunov function form a basis of the chosen
neighbourhood filter, Lyapunov decay implies convergence to the invariant
subsystem. In this way, our categorical formulation recovers the usual
Lyapunov decay criterion for convergence in familiar examples, including
autonomous discrete-time systems and guarded hybrid systems; see
Examples~\ref{ex: Lyapunov discrete autonomous}
and~\ref{ex: lyapunov guarded hybrid}.

\subsubsection*{\textbf{Acknowledgements}} The author acknowledges support by the Deutsche Forschungsgemeinschaft (DFG, German Research Foundation) under Germany's Excellence Strategy - EXC 2121 ``Quantum Universe'' - 390833306 and the Collaborative Research Center - SFB 1624 ``Higher structures, moduli spaces and integrability'' - 506632645.

\section{Categorical dynamical systems}
	
\subsection{Dynamical systems as realisations of abstract evolution shapes}

	The basic idea of a dynamical system is that of a system equipped with
	specified evolutions. We formulate this idea by distinguishing the
	\emph{syntax} of a dynamical system from its \emph{semantics}. The syntax is
	the abstract pattern of admissible evolutions: the stages, modes, or events
	one allows, together with the rules for composing them. The semantics is a
	realisation of this pattern in a chosen category of state spaces and state
	transitions. In this paper, we restrict to closed systems, namely systems with no external inputs.
	
	Let $\mathcal S$ be a small category, regarded as an abstract
	evolution shape. Let $\mathcal C$ be a coefficient category, such as
	$\mathsf{Set}$ or $\mathsf{Vect}$.
	
	\begin{definition}[Categorical dynamical systems]\label{def: dynamical system}
		A $\mathcal C$-valued dynamical system of shape $\mathcal S$ is a functor
		\begin{equation}
			X \colon \mathcal S\longrightarrow \mathcal C \ .
		\end{equation}
		We denote the category of such dynamical systems by
		\begin{equation}
			\operatorname{Dyn}_{\mathcal C}(\mathcal S)
			\coloneqq
			\operatorname{Fun}(\mathcal S,\mathcal C) \ .
		\end{equation}
	\end{definition}
	
	Thus, in our language, a dynamical system is a realisation of an abstract
	pattern of admissible evolutions in a chosen coefficient category. For every
	object $s\in\mathcal S$, the object
	\begin{equation}
		X(s)\in\mathcal C
	\end{equation}
	is the concrete state space assigned to $s$. For every morphism
	\begin{equation}
		\alpha\colon s\to s'
	\end{equation}
	in $\mathcal S$, the morphism
	\begin{equation}
		X(\alpha)\colon X(s)\longrightarrow X(s')
	\end{equation}
	is the concrete state transition assigned to $\alpha$. Functoriality says
	that identity evolutions and composites of evolutions are realized by
	identity maps and composites of transition maps:
	\begin{equation}
		X(\id_s)=\id_{X(s)}
	\end{equation}
	and
	\begin{equation}
		X(\beta\circ\alpha)=X(\beta)\circ X(\alpha)
	\end{equation}
	whenever
	\begin{equation}
		s\xrightarrow{\alpha}s'\xrightarrow{\beta}s''
	\end{equation}
	are composable morphisms in $\mathcal S$.

	A morphism between two dynamical systems
	\begin{equation}
		X,Y\colon \mathcal S\to\mathcal C
	\end{equation}
	is a natural transformation
	\begin{equation}
		\eta\colon X\Longrightarrow Y \ .
	\end{equation}
	Explicitly, it is a family of morphisms
	\begin{equation}
		\eta_s\colon X(s)\longrightarrow Y(s) \ ,
		\qquad s\in\mathcal S \ ,
	\end{equation}
	such that for every abstract evolution $\alpha\colon s\to s'$ in
	$\mathcal S$, the diagram
	\begin{equation}
		\begin{tikzcd}
			X(s) \arrow[r,"X(\alpha)"] \arrow[d,"\eta_s"'] 
			& X(s') \arrow[d,"\eta_{s'}"] \\
			Y(s) \arrow[r,"Y(\alpha)"'] 
			& Y(s')
		\end{tikzcd}
	\end{equation}
	commutes. Hence a morphism of categorical dynamical systems is a comparison
	between two realisations of the same abstract evolution category.

	\begin{definition}[Subsystems]
		Let $X\colon \mathcal S\to\mathcal C$ be a categorical dynamical system. A subsystem of $X$ is a pair $(U, \eta)$ consisting of a functor
		\begin{equation}
			V\colon \mathcal S\to\mathcal C
		\end{equation}
		and a natural transformation
		\begin{equation}
			\eta\colon V\Longrightarrow X
		\end{equation}
		such that, for every object $s\in\mathcal S$, the component
		\begin{equation}
			\eta_s\colon V(s)\longrightarrow X(s)
		\end{equation}
		is a monomorphism in $\mathcal C$.
	\end{definition}

	\begin{definition}[States, subspaces, and generalised states]
		Assume that $\mathcal C$ has a terminal object $1$, and let
		\begin{equation}
			X\colon \mathcal S\to\mathcal C
		\end{equation}
		be a categorical dynamical system.
		
		A state of $X$ over $s\in\mathcal S$ is a morphism
		\begin{equation}
			x\colon 1\longrightarrow X(s) \ .
		\end{equation}
		A subspace of $X$ over $s\in\mathcal S$ is a monomorphism
		\begin{equation}
			u\colon U\hookrightarrow X(s) \ .
		\end{equation}
		More generally, for any object $A\in\mathcal C$, an $A$-generalised state
		of $X$ over $s\in\mathcal S$ is a morphism
		\begin{equation}
			a\colon A\longrightarrow X(s) \ .
		\end{equation}
	\end{definition}

	Thus a state is a $1$-generalised state, while a subspace is a
	generalised state whose structure morphism is monic.

	\subsection{Monoid-type systems}\label{sec: monoid type}
	
	We first consider the case where the abstract system category has one
	object. Such a category is of the form
	\begin{equation}
		BM
	\end{equation}
	for a monoid $M$. A functor
	\begin{equation}
		X\colon BM\to\mathcal C
	\end{equation}
	is equivalently a single object $X_*\in\mathcal C$ equipped with a monoid action
	\begin{equation}
		M\longrightarrow \operatorname{End}_{\mathcal C}(X_*) \ .
	\end{equation}
	Dynamically, this means that all admissible
	evolutions act on one fixed state object.

	\begin{example}[Discrete-time autonomous systems]\label{ex: Discrete-time autonomous systems}
		Let $(\N,+,0)$ be the monoid of natural numbers with addition, and let
		$B\N$ denote the associated one-object category. The category $B\N$ has one
		object $*$, and
		\begin{equation}
			\operatorname{Hom}_{B\N}(*,*)=\N \ .
		\end{equation}
		The morphism $1\in\N$ represents one elementary time step, and $n\in\N$
		represents its $n$-fold iterate.
		
		A functor
		\begin{equation}
			X\colon B\N\to\mathcal C
		\end{equation}
		is characterized by an object
		\begin{equation}
			X_*\coloneqq X(*)
		\end{equation}
		together with an endomorphism
		\begin{equation}
			f\coloneqq X(1)\colon X_*\longrightarrow X_* \ .
		\end{equation}
		Functoriality gives
		\begin{equation}
			X(n)=f^n
		\end{equation}
		for every $n\in\N$. Thus an autonomous discrete-time system is precisely a
		representation of the one-object system category $B\N$.
	\end{example}
	
	\begin{example}[Continuous-time autonomous systems]
		Let $(\mathbb R_{\geq 0},+,0)$ be the monoid of nonnegative real numbers
		with addition, and let $B\mathbb R_{\geq 0}$ denote the associated
		one-object category. A functor
		\begin{equation}
			X\colon B\mathbb R_{\geq 0}\to\mathcal C
		\end{equation}
		is characterized by an object
		\begin{equation}
			X_*\coloneqq X(*)
		\end{equation}
		together with a monoid homomorphism
		\begin{equation}
			\mathbb R_{\geq 0}
			\longrightarrow
			\operatorname{End}_{\mathcal C}(X_*) \ ,
			\qquad
			t\longmapsto X(t) \ .
		\end{equation}
		Equivalently, it is a family of endomorphisms
		\begin{equation}
			\{\varphi_t\colon X_*\to X_*\}_{t\geq 0}
		\end{equation}
		satisfying
		\begin{equation}
			\varphi_0=\id_{X_*}
		\end{equation}
		and
		\begin{equation}
			\varphi_{t+s}=\varphi_t\circ\varphi_s
		\end{equation}
		for every $s,t\geq 0$. This is the categorical form of a continuous-time
		autonomous semiflow.
		
		If one wants two-sided flows, one replaces $\mathbb R_{\geq 0}$ by the
		additive group $\mathbb R$. In that case each $\varphi_t$ is an
		isomorphism, with
		\begin{equation}
			\varphi_t^{-1}=\varphi_{-t} \ .
		\end{equation}
		
		When $\mathcal C$ carries topology or smooth structure, one usually also
		requires the monoid action to have the corresponding regularity.
	\end{example}

	\begin{example}[Systems with multiple admissible updates]\label{ex: multiple admissible updates}
		Let $V$ be a finite set, and let $\overline V$ be the free monoid generated
		by $V$. The one-object category $B\overline V$ has one object $*$, and its
		endomorphisms are words in the alphabet $V$.
		
		A functor
		\begin{equation}
			X\colon B\overline V\to\mathcal C
		\end{equation}
		is characterized by an object
		\begin{equation}
			X_*\coloneqq X(*)
		\end{equation}
		together with a family of endomorphisms
		\begin{equation}
			\{U_v\colon X_*\longrightarrow X_*\}_{v\in V} \ .
		\end{equation}
		We use the convention that a word
		\begin{equation}
			w=v_1v_2\cdots v_k\in\overline V
		\end{equation}
		records the chronological order of updates. Thus $v_1$ is applied first,
		then $v_2$, and so on. Functoriality gives
		\begin{equation}
			X(w)
			=
			U_{v_k}\circ\cdots\circ U_{v_2}\circ U_{v_1} \ .
		\end{equation}
		Thus $X$ specifies a system in which several elementary updates are
		available, and a word in $\overline V$ specifies a finite sequence of
		updates.
		
		If the elementary updates satisfy relations, then one should replace the
		free monoid $\overline V$ by the category generated by the operations and relations.
		
		Many important examples in control theory are of this kind:
		\begin{enumerate}
			\item \textit{Asynchronous update systems.}
			If $V$ is the set of agents, nodes, or local components of a network, then each
			$U_v$ denotes the operation of updating only the component $v$.
			
			\item \textit{Controlled systems with finite input alphabet.}
			If $V$ is a set of control inputs, then each $U_v$ stands for the state transition encode by input $v$.
			
			\item \textit{Deterministic automata.}
			If $V$ is an input alphabet and $X_*$ is a set of automaton states, then
			the maps $U_v$ are the transition maps associated with input letters.
		\end{enumerate}
	\end{example}

	\subsection{Poset-type systems}\label{sec: poset type model}
	
	We next consider the case where the source category is a poset. When the poset is interpreted as a set of time instants, this gives a categorical model for non-autonomous systems.

	\begin{example}[Discrete-time non-autonomous systems]\label{ex: Non-autonomous discrete-time}
		Let $(\N,\leq)$ be the poset category whose objects are natural numbers,
		with a unique morphism
		\begin{equation}
			i\to j
		\end{equation}
		whenever $i\leq j$. A functor
		\begin{equation}
			X\colon (\N,\leq)\to\mathcal C
		\end{equation}
		is a categorical form of a discrete-time non-autonomous system. It consists
		of objects
		\begin{equation}
			X(0),X(1),X(2),\ldots
		\end{equation}
		and transition morphisms
		\begin{equation}
			f_{i,j}\colon X(i)\longrightarrow X(j) \ ,
			\qquad i\leq j \ ,
		\end{equation}
		satisfying
		\begin{equation}
			f_{i,i}=\id_{X(i)}
		\end{equation}
		and
		\begin{equation}
			f_{j,k}\circ f_{i,j}=f_{i,k}
		\end{equation}
		for every $i\leq j\leq k$. Equivalently, the system is determined by its
		one-step transition morphisms
		\begin{equation}
			X(0)\xrightarrow{f_{0,1}}X(1)\xrightarrow{f_{1,2}}X(2)
			\xrightarrow{f_{2,3}}\cdots \ ,
		\end{equation}
		with
		\begin{equation}
			f_{i,j}=f_{j-1,j}\circ\cdots\circ f_{i,i+1}
		\end{equation}
		for every $i<j$.
	\end{example}

	The autonomous case appears as a special realisation of the above time-indexed
	pattern. Namely, given an object $X_*\in\mathcal C$ and an endomorphism
	\begin{equation}
		f\colon X_*\longrightarrow X_* \ ,
	\end{equation}
	one obtains a functor
	\begin{equation}
		X_f\colon(\N,\leq)\to\mathcal C
	\end{equation}
	by setting
	\begin{equation}
		X_f(n)\coloneqq X_*
	\end{equation}
	for every $n\in\N$, and
	\begin{equation}
		X_f(i\leq j)\coloneqq f^{j-i}
	\end{equation}
	for every $i\leq j$. Thus the usual autonomous discrete-time system can be
	viewed as the special non-autonomous system in which all time-indexed state
	objects are identified with the same object $X_*$, and all transition
	morphisms are iterates of a single endomorphism.

	\begin{remark}
		The distinction between autonomous and non-autonomous dynamics is captured
		by the role of time. In an autonomous system, only elapsed time matters: the same time increment always induces the same
		evolution map, independently of when the evolution starts. This is reflected
		categorically by the one-object category $B\N$, whose morphisms are elapsed
		times acting as endomorphisms of a single abstract system. By contrast, in a non-autonomous system the evolution may depend on the
		starting time. This is reflected by the poset category $(\N,\leq)$, whose morphisms
		\begin{equation}
			i\to j
		\end{equation}
		record transitions from a specific starting time $i$ to a specific later
		time $j$. Thus a morphism in $(\N,\leq)$ is not determined only by the
		elapsed time $j-i$, but also by the initial time slice $i$. This is precisely
		the categorical distinction between autonomous and non-autonomous dynamics.
	\end{remark}

	\subsection{Branching, switching, and guarded evolution patterns}\label{sec: switching systems}
	
	The preceding examples have either one fixed system object or a linearly
	ordered family of time-indexed objects. Many dynamical systems have a more
	branching evolution structure. For instance, a system may have several
	modes, several admissible switches, or state-dependent guards controlling
	which transitions are allowed.

	\begin{example}[Switched systems]
		A discrete-time switched system consists of a set of modes $Q$, a state space
		\begin{equation}
			X_q
		\end{equation}
		for each mode $q\in Q$, and transition maps determined by admissible
		switches between modes. Such admissible switches can be encoded by a
		directed graph
		\begin{equation}
			G=(Q,E) \ ,
		\end{equation}
		whose vertices are the modes and whose edges are the allowed switches. An
		edge
		\begin{equation}
			e\colon q\longrightarrow q'
		\end{equation}
		is equipped with a transition map
		\begin{equation}
			F_e\colon X_q\longrightarrow X_{q'} \ .
		\end{equation}
		
		A switching signal is a directed path in $G$,
		\begin{equation}
			q_0\xrightarrow{e_0}q_1\xrightarrow{e_1}q_2
			\xrightarrow{e_2}\cdots \ ,
		\end{equation}
		and a trajectory along this switching signal is a sequence of states
		\begin{equation}
			x_n\in X_{q_n}
		\end{equation}
		satisfying
		\begin{equation}
			x_{n+1}=F_{e_n}(x_n)
		\end{equation}
		for every $n\in\N$.
		
		The passage to the categorical formulation is obtained by replacing the
		graph of admissible switches by its path category. Let
		\begin{equation}
			\mathcal P(G)
		\end{equation}
		be the path category of $G$. Its objects are the modes $q\in Q$, and its
		morphisms are finite directed paths in $G$. The switched system then defines
		a functor
		\begin{equation}
			X\colon \mathcal P(G)\longrightarrow \mathsf{Set}
		\end{equation}
		by setting
		\begin{equation}
			X(q)\coloneqq X_q
		\end{equation}
		on objects and
		\begin{equation}
			X(e)\coloneqq F_e
		\end{equation}
		on generating edges. For a path
		\begin{equation}
			q_0\xrightarrow{e_0}q_1\xrightarrow{e_1}\cdots
			\xrightarrow{e_{k-1}}q_k \ ,
		\end{equation}
		functoriality gives
		\begin{equation}
			X(e_{k-1}\cdots e_1e_0)
			=
			F_{e_{k-1}}\circ\cdots\circ F_{e_1}\circ F_{e_0} \ .
		\end{equation}
		
		Thus the path category $\mathcal P(G)$ records all finite admissible switchings, and the functor $X$ realizes each such switching as the
		corresponding state-transition map.
	\end{example}

	\begin{example}[Hybrid systems]\label{ex: hybrid guarded}
		A hybrid system has both continuous evolution and discrete jumps. It consists
		of a set of modes $Q$. For each mode $q\in Q$, the system evolves
		continuously in time, while certain discrete transitions may move the system
		from one mode to another. The important point is that jumps are usually
		\emph{guarded}: a jump from one mode to another can only happen in certain states. A categorical formulation of hybrid systems has been developed in~\cite{Am06}. Here we explain a different perspective.
		
		Let $\mathcal C$ be a category with pullbacks. We write
		\begin{equation}
			\mathsf{Par}(\mathcal C)
		\end{equation}
		for the category of partial maps in $\mathcal C$. Its objects are the
		objects of $\mathcal C$. A morphism
		\begin{equation}
			X\dashrightarrow Y
		\end{equation}
		in $\mathsf{Par}(\mathcal C)$ is represented by a span
		\begin{equation}
			X \xleftarrow{\;m\;} D \xrightarrow{\;f\;} Y
		\end{equation}
		where $m$ is a monomorphism. Two such spans represent the same partial map
		if they are isomorphic as spans over $X$ and $Y$. In the case
		$\mathcal C=\mathsf{Set}$, this is exactly a map to $Y$ defined partially on a
		subset of $X$.
		
		Let $Q$ be a set of modes, and let
		\begin{equation}
			G=(Q,E)
		\end{equation}
		be a directed graph whose edges describe the admissible jumps between
		modes. We define a hybrid source category
		\begin{equation}
			\mathcal H_G
		\end{equation}
		as follows. Its objects are pairs
		\begin{equation}
			(q,t) \ ,
			\qquad q\in Q \ ,\quad t\in\mathbb R \ .
		\end{equation}
		For every mode $q\in Q$ and every $s\leq t$, there is a continuous-evolution
		morphism
		\begin{equation}
			\gamma^q_{s,t}\colon (q,s)\longrightarrow (q,t) \ ,
		\end{equation}
		satisfying
		\begin{equation}
			\gamma^q_{t,t}=\id_{(q,t)}
		\end{equation}
		and
		\begin{equation}
			\gamma^q_{t,u}\circ \gamma^q_{s,t}
			=
			\gamma^q_{s,u}
		\end{equation}
		for every $s\leq t\leq u$. For every admissible jump
		\begin{equation}
			e\colon q\longrightarrow q'
		\end{equation}
		in $G$, and every time $t\in\mathbb R$, there is a jump morphism
		\begin{equation}
			\rho^e_t\colon (q,t)\longrightarrow (q',t) \ .
		\end{equation}
		The category $\mathcal H_G$ is generated by these morphisms, subject to the
		continuous-evolution relations above.
		
		A guarded hybrid system in $\mathcal C$ is a functor
		\begin{equation}
			X\colon \mathcal H_G\longrightarrow \mathsf{Par}(\mathcal C) \ .
		\end{equation}
		For each mode $q$ and time $t$, it assigns a state object
		\begin{equation}
			X(q,t)\in \mathcal C \ .
		\end{equation}
		A continuous-evolution morphism is realized as a partial map
		\begin{equation}
			X(\gamma^q_{s,t})\colon X(q,s)\dashrightarrow X(q,t) \ .
		\end{equation}
		When the continuous evolution in mode $q$ from time $s$ to time $t$ is
		globally defined, we denote the corresponding transition map by
		\begin{equation}
			\Phi^q_{s,t}\colon X(q,s)\longrightarrow X(q,t) \ .
		\end{equation}
		In this case, the partial map $X(\gamma^q_{s,t})$ is represented by the
		total span
		\begin{equation}
			X(q,s)
			\xleftarrow{\;\id\;}
			X(q,s)
			\xrightarrow{\;\Phi^q_{s,t}\;}
			X(q,t) \ .
		\end{equation}
		A jump morphism
		\begin{equation}
			\rho^e_t\colon (q,t)\longrightarrow (q',t)
		\end{equation}
		is realized as a partial map
		\begin{equation}
			X(\rho^e_t)\colon X(q,t)\dashrightarrow X(q',t) \ ,
		\end{equation}
		represented by a span
		\begin{equation}
			X(q,t)
			\xleftarrow{\;g^e_t\;}
			G^e_t
			\xrightarrow{\;R^e_t\;}
			X(q',t) \ .
		\end{equation}
		Here
		\begin{equation}
			g^e_t\colon G^e_t\hookrightarrow X(q,t)
		\end{equation}
		is a monomorphism, interpreted as the guard of the jump, and
		\begin{equation}
			R^e_t\colon G^e_t\longrightarrow X(q',t)
		\end{equation}
		is the reset map. Thus the jump $e$ is defined precisely on the guard object $G^e_t$.
	\end{example}

	\subsection{Stochastic systems}\label{sec: stochastic models}
	
	The preceding examples are all deterministic: an evolution sends each state to a
	uniquely determined next state. To describe probabilistic evolution, we have to replace the coefficient category by a stochastic one.
	
	A natural choice is the category
	\begin{equation}
		\mathsf{Stoch} \ .
	\end{equation}
	Its objects are measurable spaces
	\begin{equation}
		(X,\Sigma_X) \ ,
	\end{equation}
	where $X$ is a set, and $\Sigma_X$ is $\sigma$-algebra of measurable subsets. A morphism
	\begin{equation}
		K\colon (X,\Sigma_X)\longrightarrow (Y,\Sigma_Y)
	\end{equation}
	is a Markov kernel, i.e. for each
	$x\in X$, we have a probability measure
	\begin{equation}
		K(x,-)
	\end{equation}
	on $(Y,\Sigma_Y)$ such thatfor every measurable subset
	\begin{equation}
		B\in\Sigma_Y \ ,
	\end{equation}
	the function
	\begin{equation}
		X\longrightarrow [0,1] \ ,
		\qquad
		x\longmapsto K(x,B)
	\end{equation}
	is measurable. The number
	\begin{equation}
		K(x,B)
	\end{equation}
	is interpreted as the probability that a system starting at $x$ moves into
	the measurable set $B$.
	
	Composition in $\mathsf{Stoch}$ is given by integration of kernels. The composite of two morphisms
	\begin{equation}
		K\colon (X, \Sigma_X)\longrightarrow (Y,\Sigma_Y) \ , \qquad  L\colon (Y,\Sigma_Y) \longrightarrow (Z,\Sigma_Z) \ ,
	\end{equation}
	is defined by, for every $ C\in\Sigma_Z$,
	\begin{equation}
		(L\circ K)(x,C)
		=
		\int_Y L(y,C)\,K(x,dy) \ .
	\end{equation}
	
	The identity morphism on $(X,\Sigma_X)$ is the Dirac kernel
	\begin{equation}
		\id_X(x,A)=\delta_x(A) \ ,
		\qquad A\in\Sigma_X \ .
	\end{equation}
	Here $\delta_x$ is the Dirac probability measure concentrated at $x$, namely
	\begin{equation}
		\delta_x(A)
		=
		\begin{cases}
			1, & x\in A \ ,\\
			0, & x\notin A \ .
		\end{cases}
	\end{equation}

	Accordingly, a stochastic categorical dynamical system of shape
	$\mathcal S$ is a functor
	\begin{equation}
		X\colon \mathcal S\longrightarrow \mathsf{Stoch} \ .
	\end{equation}
	
	\begin{example}[Markov chains]
		A time-homogeneous Markov chain consists of a measurable state space
		\begin{equation}
			(E,\Sigma_E)
		\end{equation}
		together with a transition kernel
		\begin{equation}
			P\colon E\longrightarrow E \ .
		\end{equation}
		For each state $x\in E$ and each measurable subset
		\begin{equation}
			A\in\Sigma_E \ ,
		\end{equation}
		the number
		\begin{equation}
			P(x,A)
		\end{equation}
		is the probability that the next state lies in $A$, given that the current
		state is $x$.
		
		Categorically, this is precisely a $\mathsf{Stoch}$-valued dynamical system of shape $B\N$
		\begin{equation}
			X\colon B\N\longrightarrow \mathsf{Stoch}
		\end{equation}
		such that
		\begin{equation}
			X(*)= (E, \Sigma_E)
		\end{equation}
		and
		\begin{equation}
			X(1)=P \ .
		\end{equation}
		
		The terminal object of $\mathsf{Stoch}$ is the one-point measurable space, which we denote by $1$. A state of $(E,\Sigma_E)$ is a morphism
		\begin{equation}
			\mu\colon 1\longrightarrow (E,\Sigma_E)
		\end{equation}
		in $\mathsf{Stoch}$, which is equivalently a probability measure on $(E,\Sigma_E)$. Thus the categorical notion of state recovers the usual notion of a stochastic state.
	\end{example}

	\section{Invariant subsystems, equilibria, and orbits}

	\subsection{Invariant subsystems}

	Let
	\begin{equation}
		\Delta \colon \mathcal C \to \operatorname{Fun}(\mathcal S,\mathcal C)
	\end{equation}
	denote the constant functor. More precisely, for every $A\in\mathcal C$,
	every $s\in\mathcal S$, and every morphism $\alpha\colon s\to s'$ in
	$\mathcal S$, we have
	\begin{equation}
		(\Delta A)(s)=A \ ,
		\qquad
		(\Delta A)(\alpha)=\id_A \ .
	\end{equation}
	
	\begin{definition}[Invariant subsystems]\label{def: invariant subsystem}
		Let
		\begin{equation}
			X\colon \mathcal S\to\mathcal C
		\end{equation}
		be a categorical dynamical system. An invariant subsystem of $X$ is a pair
		$(I,\iota)$, where $I\in\mathcal C$ and
		\begin{equation}
			\iota\colon \Delta I\Longrightarrow X
		\end{equation}
		is a natural transformation whose components
		\begin{equation}
			\iota_s\colon I\longrightarrow X(s)
		\end{equation}
		are monomorphisms in $\mathcal C$.
	\end{definition}
	
	Explicitly, an invariant subsystem is a family of monomorphisms
	\begin{equation}
		\iota_s\colon I\hookrightarrow X(s) \ ,
		\qquad s\in\mathcal S \ ,
	\end{equation}
	such that for every morphism $\alpha\colon s\to s'$ in $\mathcal S$, one has
	\begin{equation}
		X(\alpha)\circ \iota_s=\iota_{s'} \ .
	\end{equation}
	Thus the same object $I$ is embedded in each stage $X(s)$, and every admissible transition carries the copy of $I$ in $X(s)$ to the copy of $I$ in $X(s')$.
	
	\begin{definition}[Invariant state]
		Assume that $\mathcal C$ has a terminal object $1$. An invariant state of
		$X$ is a natural transformation
		\begin{equation}
			x\colon \Delta 1\Longrightarrow X \ .
		\end{equation}
	\end{definition}
	
	Since $1$ is terminal, every morphism out of $1$ is a monomorphism.
	Therefore an invariant state is an invariant subsystem whose underlying
	object is $1$.

	\medskip

	In many situations, the systems $X(s)$ are all isomorphic, and one wants to
	compare states lying in different systems $X(s)$ by identifying them with a
	common system. Mathematically, this amounts to choosing an object $M\in\mathcal C$ together with a family of isomorphisms
	\begin{equation}
		\{\theta_s\colon M\xrightarrow{\sim} X(s)\}_{s \in S} \ .
	\end{equation}
	We call such a choice a \textit{state-space identification} of $X$.

	\begin{definition}[Equilibrium subsystems and equilibria]
		Let $X\colon\mathcal S\to\mathcal C$ be a dynamical system equipped with a state-space identification
		\begin{equation}
			\theta_s\colon M\xrightarrow{\sim}X(s) \ ,
			\qquad s\in\mathcal S \ .
		\end{equation}
		An equilibrium subsystem of $(X,M,\theta)$ is a monomorphism
		\begin{equation}
			e\colon E\hookrightarrow M
		\end{equation}
		such that the family
		\begin{equation}
			\iota_s\coloneqq \theta_s\circ e
			\colon E\longrightarrow X(s)
		\end{equation}
		defines an invariant subsystem of $X$.
		
		If $\mathcal C$ has a terminal object $1$, an equilibrium of
		$(X,M,\theta)$ is a morphism
		\begin{equation}
			x\colon 1\longrightarrow M
		\end{equation}
		such that the family
		\begin{equation}
			x_s \coloneqq \theta_s \circ x
		\end{equation}
		defines an invariant state.
	\end{definition}
	
	
	\medskip

	We now justify the terminology of equilibrium and equilibrium subsystem.
	Suppose that
	\begin{equation}
		X(s)=M
	\end{equation}
	for every $s\in\mathcal S$, and choose the identification
	\begin{equation}
		\theta_s=\id_M \ ,
		\qquad s\in\mathcal S \ .
	\end{equation}
	Then every morphism $\alpha\colon s\to s'$ in $\mathcal S$ determines an
	endomorphism
	\begin{equation}
		X(\alpha)\colon M\longrightarrow M \ .
	\end{equation}
	An equilibrium subsystem is precisely a monomorphism
	\begin{equation}
		e\colon E\hookrightarrow M
	\end{equation}
	such that
	\begin{equation}
		X(\alpha)\circ e=e
	\end{equation}
	for every morphism $\alpha$ in $\mathcal S$. Thus every admissible evolution
	acts trivially on the subsystem $E$.
	
	In the case $\mathcal C=\mathsf{Set}$, this means that an equilibrium
	subsystem is a subset
	\begin{equation}
		E\subseteq M
	\end{equation}
	such that
	\begin{equation}
		X(\alpha)(e)=e
	\end{equation}
	for every $e\in E$ and every morphism $\alpha$ in $\mathcal S$. Similarly,
	an equilibrium is an element
	\begin{equation}
		x\in M
	\end{equation}
	such that
	\begin{equation}
		X(\alpha)(x)=x
	\end{equation}
	for every admissible evolution $\alpha$. This recovers the usual fixed-point interpretation of equilibrium.

	\medskip

	Assume now that $\mathcal C$ has limits of shape $\mathcal S$. Then the
	constant functor admits a right adjoint
	\begin{equation}
		\lim_{\mathcal S}\colon
		\operatorname{Fun}(\mathcal S,\mathcal C)\to\mathcal C \ .
	\end{equation}
	Equivalently, for every $A\in\mathcal C$ and every
	$X\in\operatorname{Fun}(\mathcal S,\mathcal C)$, there is a natural
	bijection
	\begin{equation}
		\operatorname{Nat}(\Delta A,X)
		\cong
		\operatorname{Hom}_{\mathcal C}
		\left(A,\lim_{\mathcal S}X\right) \ .
	\end{equation}
	Denote the limiting cone by
	\begin{equation}
		\lambda_s\colon \lim_{\mathcal S}X\longrightarrow X(s) \ ,
		\qquad s\in\mathcal S \ .
	\end{equation}
	The adjunction says precisely that every natural transformation
	\begin{equation}
		a\colon \Delta A\Longrightarrow X
	\end{equation}
	factors uniquely through this limiting cone.
	
	In particular, every invariant subsystem
	\begin{equation}
		\iota\colon \Delta I\Longrightarrow X
	\end{equation}
	determines a unique morphism
	\begin{equation}
		u\colon I\longrightarrow \lim_{\mathcal S}X
	\end{equation}
	such that
	\begin{equation}
		\iota_s=\lambda_s\circ u
	\end{equation}
	for every $s\in\mathcal S$. If $\mathcal S$ has at least one object, then
	$u$ is a monomorphism. Thus every invariant subsystem determines a monomorphism
	\begin{equation}
		u\colon I\hookrightarrow \lim_{\mathcal S}X \ .
	\end{equation}

Conversely, a monomorphism $u\colon I\hookrightarrow \lim_{\mathcal S}X$ determines an invariant subsystem of $X$ precisely when, for every
$s\in\mathcal S$, the composite
\begin{equation}
    \lambda_s\circ u\colon I\longrightarrow X(s)
\end{equation}
is a monomorphism. We call such a monomorphism $I\hookrightarrow \lim_{\mathcal S}X$ \textit{admissible}. Hence, assuming that $\mathcal S$ is nonempty, invariant subsystems of $X$ are equivalent to admissible monomorphisms into $\lim_{\mathcal S}X$. As a special case, invariant states of $X$ are equivalent to morphisms
\begin{equation}
    1\longrightarrow \lim_{\mathcal S}X \ .
\end{equation}

	\medskip

	If $M$ is a monoid and $\mathcal S=BM$ is the associated one-object
	category, then there is only one system
	\begin{equation}
		X_*\coloneqq X(*) \ .
	\end{equation}
	In this case, invariant
	subsystems are precisely equilibrium subsystems, and invariant states are
	precisely equilibria.
	
	Moreover, the limit $\lim_{BM}X$ is automatically an equilibrium subsystem
	whenever it exists. Indeed, for a functor
	\begin{equation}
		X\colon BM\to\mathcal C \ ,
	\end{equation}
	write
	\begin{equation}
		\varphi_m\coloneqq X(m)\colon X_*\to X_* \ .
	\end{equation}
	The limiting cone has only one component,
	\begin{equation}
		\lambda\colon \lim_{BM}X\longrightarrow X_* \ ,
	\end{equation}
	satisfying
	\begin{equation}
		\varphi_m\circ \lambda=\lambda
	\end{equation}
	for every $m\in M$. This morphism is monic. Indeed, if
	$r,r'\colon A\to \lim_{BM}X$ satisfy
	\begin{equation}
		\lambda\circ r=\lambda\circ r' \ ,
	\end{equation}
	then the corresponding cones from $A$ to $X$ are equal, hence
	\begin{equation}
		r=r'
	\end{equation}
	by the universal property of the limit. Thus
	\begin{equation}
		\lambda\colon \lim_{BM}X\hookrightarrow X_*
	\end{equation}
	is an equilibrium subsystem. We call it the universal equilibrium subsystem
	and write
	\begin{equation}
		\operatorname{Eq}(X)\coloneqq \lim_{BM}X \ .
	\end{equation}
	Every subobject of $\operatorname{Eq}(X)$ is again an equilibrium subsystem
	of $X$.

	\begin{example}[Invariant subsystems and equilibria for non-autonomous systems]
		Let
		\begin{equation}
			X\colon (\N,\leq)\longrightarrow \mathcal C
		\end{equation}
		be a non-autonomous discrete-time system as in
		Example~\ref{ex: Non-autonomous discrete-time}. Write
		\begin{equation}
			f_{i,j}\colon X(i)\longrightarrow X(j)
		\end{equation}
		for the transition morphism associated with $i\leq j$. Since $0$ is the initial object of $(\mathbb N,\leq)$, an invariant subsystem of $X$ is equivalently a monomorphism \begin{equation} u\colon I\hookrightarrow X(0) \end{equation} such that, for every $n\in\mathbb N$, the composite \begin{equation} f_{0,n}\circ u\colon I\longrightarrow X(n) \end{equation} is a monomorphism. Thus an invariant subsystem is an embedded part of the initial state object whose images remain embedded at all later times.

		Suppose now that the system is equipped with a state-space identification
		\begin{equation}
			\theta_n\colon M\xrightarrow{\sim} X(n) \ ,
			\qquad n\in\N \ .
		\end{equation}
		Then an equilibrium subsystem is a monomorphism
		\begin{equation}
			e\colon E\hookrightarrow M
		\end{equation}
		which is fixed by all transported transition maps
		\begin{equation}
			F_{i,j}
			\coloneqq
			\theta_j^{-1}\circ f_{i,j}\circ \theta_i
			\colon M\longrightarrow M \ .
		\end{equation}
		Equivalently, for every $i\leq j$, one has
		\begin{equation}
			F_{i,j}\circ e=e \ .
		\end{equation}
		
		Now let $\mathcal C=\mathsf{Set}$ and consider the special case where all state spaces are equal to a fixed set $X_*$, and the transition maps are the iterates of a fixed map
		\begin{equation}
			f\colon X_*\longrightarrow X_* \ .
		\end{equation}
		With the canonical identification
		\begin{equation}
			M=X_* \ ,
			\qquad
			\theta_n=\id_{X_*} \ ,
		\end{equation}
		an invariant subsystem is a subset
		\begin{equation}
			I\subseteq X_*
		\end{equation}
		such that
		\begin{equation}
			f^n|_I\colon I\longrightarrow X_*
		\end{equation}
		is injective for every $n\in\N$. Equivalently, distinct elements of $I$
		have forward trajectories which never merge. In particular, every element
		\begin{equation}
			x\in X_*
		\end{equation}
		defines an invariant state, namely the trajectory
		\begin{equation}
			x,\ f(x),\ f^2(x),\ldots \ .
		\end{equation}
		By contrast, an equilibrium is an element $x\in X_*$
		satisfying
		\begin{equation}
			f(x)=x \ .
		\end{equation}
		Thus invariant states are trajectories, whereas equilibria are constant trajectories.
	\end{example}

	\subsection{Orbits}
	
	Dually, if $\mathcal C$ has colimits of shape $\mathcal S$, then the
	constant functor admits a left adjoint
	\begin{equation}
		\operatorname{colim}_{\mathcal S}\colon
		\operatorname{Fun}(\mathcal S,\mathcal C)\to\mathcal C \ .
	\end{equation}
	It is characterized by the natural bijection
	\begin{equation}
		\operatorname{Hom}_{\mathcal C}
		\left(\operatorname{colim}_{\mathcal S}X,A\right)
		\cong
		\operatorname{Nat}(X,\Delta A) \ .
	\end{equation}

	\begin{definition}[Object of orbits]
		Assume that $\mathcal C$ has colimits of shape $\mathcal S$. The object of
		orbits of a categorical dynamical system
		\begin{equation}
			X\colon \mathcal S\to\mathcal C
		\end{equation}
		is defined by
		\begin{equation}
			\operatorname{Orb}(X)\coloneqq \operatorname{colim}_{\mathcal S}X \ .
		\end{equation}
		It comes equipped with a universal cocone
		\begin{equation}
			q_s\colon X(s)\longrightarrow \operatorname{Orb}(X) \ ,
			\qquad s\in\mathcal S \ ,
		\end{equation}
		satisfying
		\begin{equation}
			q_{s'}\circ X(\alpha)=q_s
		\end{equation}
		for every morphism $\alpha\colon s\to s'$ in $\mathcal S$.
	\end{definition}

	\begin{definition}[Orbit classes]
		Let
		\begin{equation}
			a\colon A\longrightarrow X(s)
		\end{equation}
		be an $A$-generalised state of $X$ over $s\in\mathcal S$. Its orbit class is
		the composite
		\begin{equation}
			q_s\circ a\colon A\longrightarrow \operatorname{Orb}(X) \ .
		\end{equation}
		In particular, the orbit class of a subspace is defined by the same formula.
		
		If $\mathcal C$ has a terminal object $1$ and
		\begin{equation}
			x\colon 1\longrightarrow X(s)
		\end{equation}
		is a state of $X$ over $s$, then the orbit of $x$ is the composite
		\begin{equation}
			1\xrightarrow{x}X(s)
			\xrightarrow{q_s}
			\operatorname{Orb}(X) \ .
		\end{equation}
	\end{definition}

	Let $M$ be a monoid, let $\mathcal S=BM$, and let
	\begin{equation}
		X\colon BM\to\mathsf{Set}
	\end{equation}
	be given by
	\begin{equation}
		X_*\coloneqq X(*) \ ,
		\qquad
		\varphi_m\coloneqq X(m)\colon X_*\to X_* \ .
	\end{equation}
	Then the set of orbits is the quotient
	\begin{equation}
		\operatorname{Orb}(X)
		=
		X_* /{\sim} \ ,
	\end{equation}
	where $\sim$ is the equivalence relation generated by
	\begin{equation}
		x\sim \varphi_m(x)
	\end{equation}
	for all $x\in X_*$ and $m\in M$. The universal cocone is the quotient map
	\begin{equation}
		q_*\colon X_*\longrightarrow X_* /{\sim} \ ,
		\qquad
		x\longmapsto [x] \ .
	\end{equation}
	This recovers the usual notion of orbits.

	\begin{example}[Orbits for non-autonomous systems]
		Let
		\begin{equation}
			X\colon (\N,\leq)\longrightarrow \mathsf{Set}
		\end{equation}
		be a non-autonomous discrete-time system as in
		Example~\ref{ex: Non-autonomous discrete-time} with $\mathcal C = \mathsf{Set}$. In this case, the set of orbits is
		\begin{equation}
			\operatorname{Orb}(X)
			= \left(\coprod_{n\in\N}X(n)\right)\big/{\sim} \ .
		\end{equation}
		The coproduct is given by the disjoint union
		\begin{equation}
			\coprod_{n\in\N}X(n)
			=
			\{(n,x)\mid n\in\N,\ x\in X(n)\} \ .
		\end{equation}
		The equivalence relation $\sim$ is generated by
		\begin{equation}
			(i,x)\sim (j,f_{i,j}(x))
		\end{equation}
		for every $i\leq j$ and every $x\in X(i)$. Thus a state is identified with
		all of its future images. The universal cocone is given by the maps
		\begin{equation}
			q_n\colon X(n)\longrightarrow \operatorname{Orb}(X) \ ,
			\qquad
			x\longmapsto [(n,x)] \ .
		\end{equation}
	\end{example}

	\subsection{Convergence}
	
	We now discuss convergence. Informally, convergence expresses the idea that
	an initial state eventually exhibits a prescribed behaviour. To formulate
	this idea in our setting, we follow our general principle of separating
	syntax from semantics. The initial state and the prescribed behaviour belong
	to the concrete realisation, whereas the notion of ``eventually'' should be
	understood as additional asymptotic structure on the abstract evolution
	shape. We therefore begin by introducing a categorical structure on
	$\mathcal S$ which makes the notion of eventuality precise.
	
	\medskip
	
	Let $s\in\mathcal S$ be an object. We denote by
	\begin{equation}
		\operatorname{Out}(s)
		=
		\{\,\alpha\colon s\to t \mid t\in\mathcal S\,\}
	\end{equation}
	the set of all morphisms with source $s$. We regard $\operatorname{Out}(s)$ as the set of all abstract evolutions starting at $s$.
	
	There is a natural preorder on $\operatorname{Out}(s)$. For
	\begin{equation}
		\alpha\colon s\to t
		\qquad
		\text{and}
		\qquad
		\alpha'\colon s\to t' \ ,
	\end{equation}
	we write
	\begin{equation}
		\alpha\preceq \alpha'
	\end{equation}
	if there exists a morphism $ \beta\colon t\to t'$ such that
	\begin{equation}
		\alpha'=\beta\circ\alpha \ .
	\end{equation}
	
	\begin{remark}
		In general, this preorder need not be a partial order because antisymmetry may fail. For example, if $\mathcal S=BG$ is the one-object category associated with a group $G$, then every morphism can be obtained from any other by composition with another
		group element.
	\end{remark}

	The preorder $\preceq$ determines an Alexandrov topology on
	$\operatorname{Out}(s)$ by declaring the open subsets to be the
	upward-closed subsets. Thus a subset
	\begin{equation}
		R\subseteq \operatorname{Out}(s)
	\end{equation}
	is open if, whenever
	\begin{equation}
		\alpha\in R
		\qquad
		\text{and}
		\qquad
		\alpha\preceq\alpha' \ ,
	\end{equation}
	one has
	\begin{equation}
		\alpha'\in R \ .
	\end{equation}
	Equivalently, $R$ is open if it is closed under further evolution: for every
	morphism
	\begin{equation}
		(\alpha\colon s\to t)\in R
	\end{equation}
	and every morphism $\beta\colon t\to u$ in $\mathcal S$, one has
	\begin{equation}
		\beta\circ\alpha\in R \ .
	\end{equation}
	
	In categorical language, the open subsets of $\operatorname{Out}(s)$ are precisely the \textit{cosieves} on $s$. We denote by
	\begin{equation}
		\operatorname{Cos}(s)
	\end{equation}
	the set of cosieves on $s$, so $\operatorname{Cos}(s)$ is the
	topology of open subsets of $\operatorname{Out}(s)$ for the Alexandrov
	topology defined above.
	
	The set $\operatorname{Cos}(s)$ is naturally a poset under inclusion:
	\begin{equation}
		R\leq R'
		\qquad
		\text{if and only if}
		\qquad
		R\subseteq R' \ .
	\end{equation}
	The smallest element is the empty cosieve, and the largest element is
	$\operatorname{Out}(s)$ itself. Moreover, arbitrary unions and arbitrary
	intersections of cosieves are again cosieves. Indeed, if
	$\{R_i\}_{i\in I}$ is a family of cosieves and
	\begin{equation}
		\alpha\in\bigcap_{i\in I}R_i \ ,
	\end{equation}
	then every further extension of $\alpha$ belongs to every $R_i$, hence
	belongs to $\bigcap_{i\in I}R_i$. The proof for unions is similar.

	\begin{definition}[Eventuality filters]
		An eventuality filter on $\mathcal S$ assigns to every object $s\in\mathcal S$ a collection
		\begin{equation}
			\mathcal E_s\subseteq \operatorname{Cos}(s)
		\end{equation}
		of cosieves on $s$, called the eventual cosieves at $s$, such that:
		
		\begin{enumerate}
			\item $\mathcal E_s$ is nonempty;
			
			\item if $R\in\mathcal E_s$ and $R\subseteq R'$ for another cosieve
			$R'$ on $s$, then
			\begin{equation}
				R'\in\mathcal E_s \ ;
			\end{equation}
			
			\item if $R,R'\in\mathcal E_s$, then
			\begin{equation}
				R\cap R'\in\mathcal E_s \ ;
			\end{equation}
			
			\item the empty cosieve is not eventual:
			\begin{equation}
				\varnothing\notin\mathcal E_s \ .
			\end{equation}
		\end{enumerate}
	\end{definition}
	
	Equivalently, for each object $s\in\mathcal S$, the collection $\mathcal E_s$ is a filter of open subsets on $\operatorname{Out}(s)$.

	\begin{definition}[Eventual properties]
		Let $\mathcal S$ be a category equipped with an eventuality filter $\mathcal E$. Let $s\in\mathcal S$, and let $P$ be a property of morphisms
		with source $s$. We say that $P$ holds eventually from $s$ if there exists
		an eventual cosieve
		\begin{equation}
			R\in\mathcal E_s
		\end{equation}
		such that every morphism in $R$ satisfies $P$.
	\end{definition}
	
	Now let
	\begin{equation}
		X\colon \mathcal S\to\mathcal C
	\end{equation}
	be a categorical dynamical system, where $\mathcal S$ is equipped with an
	eventuality filter $\mathcal E$. We first formulate the strongest form of convergence.
	
	\begin{definition}[Eventual absorption into an invariant subsystem]
		Let
		\begin{equation}
			\iota\colon \Delta I\Longrightarrow X
		\end{equation}
		be an invariant subsystem of $X$, and let
		\begin{equation}
			a\colon A\to X(s)
		\end{equation}
		be an $A$-generalised state over $s$. We say that $a$ is eventually absorbed by $I$ if the following property holds eventually: for a morphism with source $s$
		\begin{equation}
			\alpha\colon s\to t \ ,
		\end{equation}
		the transported generalised state
		\begin{equation}
			X(\alpha)\circ a\colon A\to X(t)
		\end{equation}
		factors through
		\begin{equation}
			\iota_t\colon I\hookrightarrow X(t) \ .
		\end{equation}
	\end{definition}
	
	\begin{remark}
		Since $\iota_t$ is monic, such a factorisation is unique whenever it exists.
	\end{remark}
	
	Equivalently, $a$ is eventually absorbed by $I$ if there exists an eventual
	cosieve
	\begin{equation}
		R\in\mathcal E_s
	\end{equation}
	such that, for every morphism $\alpha\colon s\to t$ in $R$, there exists a morphism
	\begin{equation}
		\widetilde a_\alpha\colon A\to I
	\end{equation}
	satisfying
	\begin{equation}
		\iota_t\circ \widetilde a_\alpha
		=
		X(\alpha)\circ a \ .
	\end{equation}
	
	In the case $A=1$, this says that an ordinary state eventually lands in the
	invariant subsystem $I$. Since $I$ is invariant, once an evolved state lies
	in $I$, all further evolutions remain in $I$.
	
	If $X$ is equipped with a state-space identification
	\begin{equation}
		\{\theta_s\colon M\xrightarrow{\sim}X(s)\}_{s\in\mathcal S}
	\end{equation}
	and
	\begin{equation}
		e\colon E\hookrightarrow M
	\end{equation}
	is an equilibrium subsystem, then a generalised state is said to be
	eventually absorbed by $E$ if it is eventually absorbed by the associated
	invariant subsystem
	\begin{equation}
		\theta_s\circ e\colon E\hookrightarrow X(s) \ ,
		\qquad s\in\mathcal S \ .
	\end{equation}
	
	In particular, if
	\begin{equation}
		\mathrm{eq}\colon 1\to M
	\end{equation}
	is an equilibrium, then a state
	\begin{equation}
		x\colon 1\to X(s)
	\end{equation}
	is eventually absorbed by $\mathrm{eq}$ precisely when the property
	\begin{equation}
		X(\alpha)\circ x
		=
		\theta_t\circ \mathrm{eq}
	\end{equation}
	holds eventually for morphisms $\alpha\colon s\to t$ with source $s$.

	\begin{example}[Eventual absorption in autonomous systems]
		Let $\mathcal S=B\N$, and let
		\begin{equation}
			X\colon B\N\to\mathsf{Set}
		\end{equation}
		be an autonomous system determined by a map
		\begin{equation}
			f\colon X_*\to X_* \ .
		\end{equation}
		
		The nonempty cosieves on $*$ are precisely the subsets
		\begin{equation}
			R_N
			=
			\{\,n\in\N\mid n\geq N\,\}
			\subseteq \operatorname{Hom}_{B\N}(*,*)
		\end{equation}
		for some $N\in\N$. We choose the eventual filter on $B\N$ to be
		\begin{equation}
			\mathcal E_*
			=
			\{\,R_N\mid N\in\N\,\} \ .
		\end{equation}
		
		Let
		\begin{equation}
			I\subseteq X_*
		\end{equation}
		be an invariant subsystem and let
		\begin{equation}
			x_0\colon 1\to X_*
		\end{equation}
		be an initial state, regarded equivalently as an element $ x_0\in X_*$. Then $x_0$ is eventually absorbed by $I$ precisely when there exists $ N\in\N$ such that
		\begin{equation}
			f^n(x_0)\in I \ , \qquad  \forall n \geq N \ .
		\end{equation}
		Thus eventual absorption recovers the usual notion of reaching the fixed subsystem $I$ in finite time.
	\end{example}

	Eventual absorption is often too strong for what we usually mean by
	convergence. For example, a trajectory converging to a fixed point does not
	usually reach that point after finite evolution; rather, it eventually lies
	in every neighbourhood of it. Thus, besides the eventuality filter on the
	source category, we also need a filter on the semantic side.
	
	\begin{definition}[Stagewise neighbourhoods]
		Let
		\begin{equation}
			\iota\colon \Delta I\Longrightarrow X
		\end{equation}
		be an invariant subsystem of
		\begin{equation}
			X\colon \mathcal S\to\mathcal C \ .
		\end{equation}
		A stagewise neighbourhood of $I$ in $X$ is a family of monomorphisms
		\begin{equation}
			\mu_s\colon U_s\hookrightarrow X(s) \ ,
			\qquad s\in\mathcal S \ ,
		\end{equation}
		such that, for every object $s\in\mathcal S$, the component
		\begin{equation}
			\iota_s\colon I\hookrightarrow X(s)
		\end{equation}
		factors through $\mu_s$.
	\end{definition}
	
	\begin{remark}
		Since $\mu_s$ is monic, such a factorisation is unique whenever it exists. Moreover, the induced morphism
		\begin{equation}
			I\longrightarrow U_s
		\end{equation}
		is itself monic, because its composite with $\mu_s$ is the monomorphism $\iota_s$.
	\end{remark}
	
	We order stagewise neighbourhoods by componentwise inclusion. For two stagewise neighbourhoods
	\begin{equation}
		U=\big(\mu_s\colon U_s\hookrightarrow X(s)\big)_{s\in\mathcal S}
	\end{equation}
	and
	\begin{equation}
		V=\big(\nu_s\colon V_s\hookrightarrow X(s)\big)_{s\in\mathcal S} \ ,
	\end{equation}
	we write
	\begin{equation}\label{eq: preorder on Nbd}
		U\leq V
	\end{equation}
	if, for every object $s\in\mathcal S$, there exists a morphism
	\begin{equation}
		\phi_s\colon U_s\longrightarrow V_s
	\end{equation}
	making the triangle
	\begin{equation}
		\begin{tikzcd}
			U_s \arrow[rr,"\phi_s"] \arrow[dr,"\mu_s"'] && V_t \arrow[dl,"\nu_s"] \\
			& X(s) &
		\end{tikzcd}
	\end{equation}
	commute. Such a morphism $\phi_s$ is unique whenever it exists
	and is monic. This defines a preorder on the collection of stagewise
	neighbourhoods of $I$ in $X$, which we denote by
	\begin{equation}
		\big(\mathsf{Nbd}_X(I), \leq \big) \ .
	\end{equation}

	\begin{definition}[Neighbourhood filters]
		A neighbourhood filter of the invariant subsystem $I$ in $X$ is a collection
		\begin{equation}
			\mathcal N_X(I)\subseteq \mathsf{Nbd}_X(I)
		\end{equation}
		of stagewise neighbourhoods of $I$ satisfying the following conditions:
		
		\begin{enumerate}
			\item $\mathcal N_X(I)$ is nonempty;
			
			\item if $U\in\mathcal N_X(I)$ and $U\leq V$, then
			\begin{equation}
				V\in\mathcal N_X(I) \ ;
			\end{equation}
			
			\item if $U,V\in\mathcal N_X(I)$, then there exists
			$W\in\mathcal N_X(I)$ such that
			\begin{equation}
				W\leq U
				\qquad
				\text{and}
				\qquad
				W\leq V \ .
			\end{equation}
		\end{enumerate}
	\end{definition}

	If $\mathcal C$ has pullbacks, then the third condition above may be expressed
	using pullbacks as follows: for any $ U,V\in\mathcal N_X(I)$, there exists $W\in\mathcal N_X(I)$ such that, for every $t\in\mathcal S$, the monomorphism
	\begin{equation}
		W_t\hookrightarrow X(t)
	\end{equation}
	factors through the pullback square
	\begin{equation}
		\begin{tikzcd}
			U_t\times_{X(t)}V_t
			\arrow[r]
			\arrow[d]
			& V_t \arrow[d,hook,"v_t"] \\
			U_t \arrow[r,hook,"\mu_t"']
			& X(t)
		\end{tikzcd}
	\end{equation}

	\begin{definition}[Convergence to an invariant subsystem]\label{def: convergence}
		Let $X\colon\mathcal S\to\mathcal C$ be a categorical dynamical system, where $\mathcal S$ is equipped with an
		eventuality filter
		\begin{equation}
			\mathcal E=\{\mathcal E_s\}_{s\in\mathcal S} \ .
		\end{equation}
		Let $\iota\colon \Delta I\Longrightarrow X$ be an invariant subsystem equipped with a neighbourhood filter
		\begin{equation}
			\mathcal N_X(I) \ ,
		\end{equation}
		and let
		\begin{equation}
			a\colon A\to X(s)
		\end{equation}
		be an $A$-generalised state over $s$.
		
		We say that $a$ converges to $I$ if, for every neighbourhood
		\begin{equation}
			U=\bigl(\mu_s\colon U_s\hookrightarrow X(s)\bigr)_{s\in\mathcal S}
			\in \mathcal N_X(I) \ ,
		\end{equation}
		the following property holds eventually: for a
		morphism with source $s$
		\begin{equation}
			\alpha\colon s\to t \ ,
		\end{equation}
		the evolution of the generalised state
		\begin{equation}
			X(\alpha)\circ a\colon A\to X(t)
		\end{equation}
		factors through
		\begin{equation}
			\mu_t\colon U_t\hookrightarrow X(t) \ .
		\end{equation}
	\end{definition}

	More explicitly, for every
	\begin{equation}
		U\in\mathcal N_X(I) \ ,
	\end{equation}
	there exists an eventual cosieve
	\begin{equation}
		R_U\in\mathcal E_s
	\end{equation}
	such that, for every morphism $\alpha\colon s\to t$ in $R_U$, there exists a morphism
	\begin{equation}
		\widetilde a_\alpha\colon A\to U_t
	\end{equation}
	satisfying
	\begin{equation}
		\mu_t\circ \widetilde a_\alpha
		=
		X(\alpha)\circ a \ .
	\end{equation}
	In the case $A=1$, this says that an ordinary state converges to $I$ if it eventually lies in every preferred neighbourhood of $I$.
	
	Suppose now that $X$ is equipped with a state-space identification
	\begin{equation}
		\{\theta_s\colon M\xrightarrow{\sim}X(s)\}_{s\in\mathcal S}
	\end{equation}
	and that
	\begin{equation}
		e\colon E\hookrightarrow M
	\end{equation}
	is an equilibrium subsystem. Then $e$ determines the invariant subsystem
	\begin{equation}
		\theta_s\circ e\colon E\hookrightarrow X(s) \ ,
		\qquad s\in\mathcal S \ .
	\end{equation}
	After choosing a neighbourhood filter for this associated invariant
	subsystem, we say that an $A$-generalised state converges to $E$ if it converges to the associated invariant subsystem in the sense above.

	\begin{example}[Convergence in non-autonomous systems]\label{ex: convergence discrete non-auto}
		Let
		\begin{equation}
			X\colon(\N,\leq)\to\mathsf{Top}
		\end{equation}
		be a non-autonomous discrete-time system. We write
		\begin{equation}
			f_{i,j}\coloneqq X(i\leq j)\colon X(i)\to X(j)
		\end{equation}
		for $i\leq j$.
		
		For each $i\in\N$, the set $\operatorname{Out}(i)$ consists of the
		morphisms
		\begin{equation}
			i\to j \ ,
			\qquad j\geq i \ .
		\end{equation}
		For $N\geq i$, define
		\begin{equation}
			R_N^i
			=
			\{\,i\to j\mid j\geq N\,\}
			\subseteq \operatorname{Out}(i) \ .
		\end{equation}
		We take the eventuality filter on $(\N, \leq)$ to be
		\begin{equation}
			\mathcal E_i
			=
			\{\,R_N^i\mid N\geq i\,\} \ .
		\end{equation}
		
		Suppose now that $X$ is equipped with a state-space identification
		\begin{equation}
			\theta_j\colon M\xrightarrow{\sim}X(j) \ ,
			\qquad j\in\N \ .
		\end{equation}
		Let
		\begin{equation}
			x_\infty\in M
		\end{equation}
		be an equilibrium. Equivalently, the points
		\begin{equation}
			\xi_j\coloneqq \theta_j(x_\infty)\in X(j) \ ,
			\qquad j\in\N \ ,
		\end{equation}
		define an invariant state
		\begin{equation}
			\xi\colon \Delta 1\Longrightarrow X \ .
		\end{equation}
		
		We define a neighbourhood filter
		\begin{equation}
			\mathcal N_X(x_\infty)
			=
			\left\{
			U=(U_j\subseteq X(j))_{j\in\N}
			\ \middle|\
			\begin{array}{l}
				\text{there exists an open neighbourhood }O\subseteq M
				\text{ of }x_\infty\\
				\text{such that }\theta_j(O)\subseteq U_j
				\text{ for every }j\in\N
			\end{array}
			\right\} \ .
		\end{equation}
		It is straightforward to check that this is indeed a filter using the fact that $\theta_j$ are homeomorphisms.
		
		Let
		\begin{equation}
			x_i\in X(i)
		\end{equation}
		be an initial state at time $i$. Then $x_i$ converges to the equilibrium
		$x_\infty$ precisely when, for every open neighbourhood
		\begin{equation}
			O\subseteq M
		\end{equation}
		of $x_\infty$, there exists $N\geq i$ such that
		\begin{equation}
			\theta_j^{-1} \big(f_{i,j}(x_i) \big) \in O
		\end{equation}
		for every $j \geq N$. This is equivalent to
		\begin{equation}
			\theta_j^{-1}\bigl(f_{i,j}(x_i)\bigr)
			\longrightarrow
			x_\infty
		\end{equation}
		as $j\to\infty$. Hence we recover the usual
		topological convergence condition.
	\end{example}

	\section{A categorical Lyapunov principle}
	
	The classical Lyapunov method is a way of proving stability and convergence
	without explicitly solving the equations of motion. Its basic idea is to
	assign to a system an energy-like quantity which measures deviation from the
	desired states. If this quantity does not increase along evolution,
	then every sublevel set of this quantity is preserved by the dynamics.

	Note that the numerical Lyapunov function
	is not the essential structure by itself. What matters for the stability
	argument is the family of sublevel neighbourhoods that it produces, together with the fact that these neighbourhoods are respected by the admissible evolutions. This point of view is especially well suited to our categorical formulation.

    \subsection{Lyapunov stability}
	
	\begin{definition}[Lyapunov stability]\label{def: Lyapunov stable}
		Let $X\colon\mathcal S\to\mathcal C$ be a dynamical system and let
		\begin{equation}
			\iota\colon \Delta I\Longrightarrow X
		\end{equation}
		be an invariant subsystem equipped with a neighbourhood filter
		$\mathcal N_X(I)$. We say that $I$ is Lyapunov stable in $X$ if the following condition holds.
		For every neighbourhood
		\begin{equation}
			U =
			\bigl(\mu_s\colon U_s\hookrightarrow X(s)\bigr)_{s\in\mathcal S}
			\in \mathcal N_X(I) \ ,
		\end{equation}
		there exists a neighbourhood
		\begin{equation}
			V =
			\bigl(\nu_s\colon V_s\hookrightarrow X(s)\bigr)_{s\in\mathcal S}
			\in \mathcal N_X(I)
		\end{equation}
		such that
		\begin{equation}
			V\leq U
		\end{equation}
		in the sense of~\eqref{eq: preorder on Nbd}, and such that, for every
		morphism $\alpha\colon s\to t$ in $\mathcal S$, the composite
		\begin{equation}
			V_s\xrightarrow{\nu_s}X(s)\xrightarrow{X(\alpha)}X(t)
		\end{equation}
		factors through
		\begin{equation}
			\mu_t\colon U_t\hookrightarrow X(t) \ .
		\end{equation}
	\end{definition}

	\begin{definition}[Forward-invariant neighbourhoods]\label{def: forward invariant neighbourhoods}
		Let $\iota\colon \Delta I\Rightarrow X$ be an invariant subsystem. A stagewise neighbourhood
		\begin{equation}
			U=\big(\mu_s\colon U_s\hookrightarrow X(s)\big)_{s\in\mathcal S}
		\end{equation}
		of $I$ is called forward-invariant if, for every morphism $\alpha\colon s\to t$ in $\mathcal S$, the composite
		\begin{equation}
			U_s\xrightarrow{\mu_s}X(s)\xrightarrow{X(\alpha)}X(t)
		\end{equation}
		factors through
		\begin{equation}
			\mu_t\colon U_t\hookrightarrow X(t) \ .
		\end{equation}
		Equivalently, there exists a morphism
		\begin{equation}
			U(\alpha)\colon U_s\longrightarrow U_t
		\end{equation}
		such that
		\begin{equation}
			\mu_t\circ U(\alpha)=X(\alpha)\circ \mu_s \ .
		\end{equation}
	\end{definition}
	
	Since $\mu_t$ is monic, the morphism $U(\alpha)$ is unique whenever it
	exists. It is straightforward to verify that we therefore obtain a functor
	\begin{equation}
		U\colon \mathcal S\to\mathcal C \ , \qquad s \mapsto U_s\ , \qquad \alpha \mapsto U(\alpha) \ .
	\end{equation}
	Thus a forward-invariant neighbourhood is precisely a subsystem of $X$ which is also a stagewise neighbourhood of $I$.
	
	\begin{definition}[Lyapunov bases]\label{def: lyapunov basis}
		Let $\iota\colon \Delta I\Longrightarrow X$ be an invariant subsystem equipped with a neighbourhood filter $\mathcal N_X(I)$.
		
		A Lyapunov basis for $\mathcal N_X(I)$ is a subcollection
		\begin{equation}
			\mathcal B\subseteq \mathcal N_X(I)
		\end{equation}
		such that:
		
		\begin{enumerate}
			\item every $V\in \mathcal B$ is forward-invariant;
			
			\item $\mathcal B$ is a basis for the filter $\mathcal N_X(I)$, in the sense that for every $U \in \mathcal N_X(I)$,
			there exists $V \in \mathcal B$
			such that
			\begin{equation}
				V \leq U \ .
			\end{equation}
		\end{enumerate}
	\end{definition}

	\begin{lemma}\label{prop: Lyapunov stability}
		Let
		\begin{equation}
			X\colon\mathcal S\to\mathcal C
		\end{equation}
		be a categorical dynamical system, and let
		\begin{equation}
			\iota\colon \Delta I\Longrightarrow X
		\end{equation}
		be an invariant subsystem equipped with a neighbourhood filter
		$\mathcal N_X(I)$. If $I$ admits a Lyapunov neighbourhood basis, then $I$ is
		Lyapunov stable in $X$.
	\end{lemma}
	
	\begin{proof}
		Let
		\begin{equation}
			U \in \mathcal N_X(I)
		\end{equation}
		be an arbitrary neighbourhood of $I$. By our assumption, there exists $V \in \mathcal N_X(I)$ such that $V$ is forward-invariant and $V \leq U$. We claim that $V$ satisfies the stability condition in Definition~\ref{def: Lyapunov stable}.
		
		Indeed, let
		\begin{equation}
			\alpha\colon s\to t
		\end{equation}
		be a morphism in $\mathcal S$. Since $V$ is forward-invariant, the
		composite
		\begin{equation}
			V_s\xrightarrow{\nu_s}X(s)\xrightarrow{X(\alpha)}X(t)
		\end{equation}
		factors through
		\begin{equation}
			\nu_t\colon V_t\hookrightarrow X(t) \ .
		\end{equation}
		Since $V\leq U$, the monomorphism $\nu_t$ factors through
		\begin{equation}
			\mu_t\colon U_t\hookrightarrow X(t) \ .
		\end{equation}
		Therefore, $V_s\longrightarrow X(t)$ also factors through $U_t$. This proves Lyapunov stability.
	\end{proof}

\subsection{Categorical Lyapunov functions}

We now introduce categorical Lyapunov functions and explain how thy produce Lyapunov bases. We assume that $\mathcal C$ has terminal object $1$ and pullbacks.

\begin{definition}[Lyapunov measurement data]\label{def: lyapunov measurement}
A Lyapunov
measurement datum in $\mathcal C$ consists of:

\begin{enumerate}
    \item an object
    \begin{equation}
        R\in\mathcal C \ ,
    \end{equation}
    called the measurement object;

    \item a distinguished morphism
    \begin{equation}
        0\colon 1\to R \ ,
    \end{equation}
    called the zero value;

    \item a family of monomorphisms
    \begin{equation}
        r_\lambda\colon R_\lambda\hookrightarrow R \ ,
        \qquad \lambda\in\Lambda
    \end{equation}
    for some index family $\Lambda$, interpreted as sublevel neighbourhoods of $0$;

    \item for every object $Y\in\mathcal C$, a preorder on
    \begin{equation}
        \operatorname{Hom}_{\mathcal C}(Y,R)
    \end{equation}
    which is preserved by precomposition.
\end{enumerate}
These data are required to satisfy the following two conditions: first, the 0-value morphism factors through $r_\lambda$ for every $\lambda \in \Lambda$; second, for any
\begin{equation}
    h\leq k\colon Y\to R \ ,
\end{equation}
if $k$ factors through $r_\lambda$, then $h$ also factors through $r_\lambda$.
\end{definition}

\begin{example}
Let $ \mathcal C=\mathsf{Set}$. Take
\begin{equation}
    R=\mathbb R_{\geq 0} \ ,
\end{equation}
and
\begin{equation}
    0\colon *\to \mathbb R_{\geq 0}
\end{equation}
be the point $0$. For every $\varepsilon>0$, take $r_\epsilon$ to be the canonical inclusion
\begin{equation}
    R_\varepsilon=[0,\varepsilon]
    \hookrightarrow
    \mathbb R_{\geq 0} \ .
\end{equation}

For any set $Y$, the preorder on $\operatorname{Hom}_{\mathsf{Set}}(Y,\mathbb R_{\geq 0})$ is the pointwise order: for maps
\begin{equation}
    h,k\colon Y\to \mathbb R_{\geq 0} \ ,
\end{equation}
we write
\begin{equation}
    h\leq k
\end{equation}
if and only if $h(y)\leq k(y)$ for every $y\in Y$. This preorder is preserved by precomposition. Indeed, if $f\colon Z\to Y$ is another map and $h\leq k$, then we have
\begin{equation}
    h\circ f\leq k\circ f \ .
\end{equation}

Moreover, it is immediate to see that every $[0,\varepsilon]$ contains $0$ and is downward closed with respect to the preorder on $\operatorname{Hom}_{\mathsf{Set}}(Y,\mathbb R_{\geq 0})$.
\end{example}

\begin{definition}[Categorical Lyapunov functions]\label{def: lyapunov function}
Let $ X\colon \mathcal S\to\mathcal C$ be a dynamical system, and let $ \iota\colon \Delta I\Rightarrow X$ be an invariant subsystem equipped with a neighbourhood filter $ \mathcal N_X(I)$.

A categorical Lyapunov function for $I$ consists of a Lyapunov measurement
datum
\begin{equation}
    \mathfrak R=(R,0,\{r_\lambda\colon R_\lambda\hookrightarrow R\}_{\lambda\in\Lambda},\leq)
\end{equation}
in $\mathcal C$, together with a family of morphisms
\begin{equation}
    L_s\colon X(s)\to R \ ,
    \qquad s\in\mathcal S \ ,
\end{equation}
satisfying the following conditions:

\begin{enumerate}
    \item $L$ vanishes on the invariant subsystem $I$: for every
    $s\in\mathcal S$,
    \begin{equation}
        L_s\circ \iota_s
        =
        0\circ {!}_I \ ,
    \end{equation}
    where ${!}_I$ is the unique morphism $I \to 1$.

    \item $L$ is non-increasing along admissible evolutions: for every
    morphism $\alpha\colon s\to t$ in $\mathcal S$,
    \begin{equation}
        L_t\circ X(\alpha)\leq L_s
    \end{equation}
    as morphisms $X(s)\to R$.

    \item The sublevel stagewise neighbourhoods
    \begin{equation}
        U^\lambda_s
        \coloneqq
        X(s)\times_R R_\lambda
    \end{equation}
    form a basis for the neighbourhood filter $\mathcal N_X(I)$.
\end{enumerate}
\end{definition}

\begin{remark}
    We clarify here why the objects $U^\lambda_s$ are indeed stagewise
neighbourhoods of $I$. For each $\lambda\in\Lambda$ and
$s\in\mathcal S$, the object $U^\lambda_s$ is defined as the pullback of the
monomorphism $r_\lambda\colon R_\lambda\hookrightarrow R$ along
$L_s\colon X(s)\to R$. Since monomorphisms are stable under pullback, the
projection
\begin{equation}
    \mu^\lambda_s\colon U^\lambda_s\hookrightarrow X(s)
\end{equation}
is a monomorphism.

Moreover, since $L_s\circ\iota_s=0\circ {!}_I$ and since $0\colon 1\to R$
factors through every $r_\lambda\colon R_\lambda\hookrightarrow R$, the
composite $ L_s\circ\iota_s\colon I\to R$ factors through $R_\lambda$. Hence the pair of morphisms
\begin{equation}
    \iota_s\colon I\to X(s),
    \qquad
    I\to R_\lambda
\end{equation}
satisfies the pullback compatibility condition. By the universal property of
the pullback, there is therefore a unique morphism
\begin{equation}
    I\longrightarrow U^\lambda_s
\end{equation}
whose composite with $\mu_s^\lambda$ is $\iota_s$. Consequently, for every $\lambda \in \Lambda$,
\begin{equation}
    U^\lambda=
    \bigl(u^\lambda_s\colon U^\lambda_s\hookrightarrow X(s)\bigr)_{s\in\mathcal S}
\end{equation}
is a stagewise neighbourhood of $I$.
\end{remark}

We now prove that a categorical Lyapunov function induces a Lyapunov basis.

\begin{theorem}\label{thm: lyapunov function produces Lyapunov basis}
Let
\begin{equation}
    L=\{L_s\colon X(s)\to R\}_{s\in\mathcal S}
\end{equation}
be a categorical Lyapunov function for $I$. Then the associated sublevel
neighbourhoods
\begin{equation}
    U^\lambda=
    \bigl(U^\lambda_s\hookrightarrow X(s)\bigr)_{s\in\mathcal S} \ ,
    \qquad \lambda\in\Lambda \ ,
\end{equation}
form a Lyapunov basis for $\mathcal N_X(I)$. Consequently, $I$ is Lyapunov
stable in $X$.
\end{theorem}

\begin{proof}
By definition of a categorical Lyapunov function, the sublevel
neighbourhoods $U^\lambda$ form a basis for $\mathcal N_X(I)$. It remains to
show that each $U^\lambda$ is forward-invariant.

Let $\alpha\colon s\to t$ be a morphism in $\mathcal S$. We must show that
the composite
\begin{equation}
    U^\lambda_s\hookrightarrow X(s)\xrightarrow{X(\alpha)}X(t)
\end{equation}
factors through
\begin{equation}
    U^\lambda_t\hookrightarrow X(t) \ .
\end{equation}
By construction, the composite
\begin{equation}
    U^\lambda_s\hookrightarrow X(s)\xrightarrow{L_s}R
\end{equation}
factors through
\begin{equation}
    R_\lambda\hookrightarrow R \ .
\end{equation}
The Lyapunov inequality gives
\begin{equation}
    L_t\circ X(\alpha)\leq L_s \ .
\end{equation}
Since $R_\lambda$ is downward closed, it follows that
\begin{equation}
    U^\lambda_s\hookrightarrow X(s)
    \xrightarrow{X(\alpha)}X(t)
    \xrightarrow{L_t}R
\end{equation}
also factors through $R_\lambda$. By the pullback definition of
$U^\lambda_t$, this induces a unique morphism $\lambda_s^t\colon U_s^\lambda \to U_t^\lambda$ such that
\begin{equation}
    X(\alpha) \circ \mu_s^\lambda = \mu_\lambda^t \circ \lambda_s^t \ .
\end{equation}
Hence $U^\lambda$ is forward-invariant.

Therefore the sublevel neighbourhoods form a Lyapunov basis. The final claim
follows from Lemma~\ref{prop: Lyapunov stability}.
\end{proof}

\subsection{Convergence via Lyapunov decay}

Lyapunov stability says that sufficiently small perturbations remain small
under all future evolutions. To prove actual convergence to an invariant
subsystem, one needs an additional decay condition.

\begin{definition}[Lyapunov decay]
Let
\begin{equation}
    L=\{L_s\colon X(s)\to R\}_{s\in\mathcal S}
\end{equation}
be a Lyapunov function for an invariant subsystem $I$ of a dynamical
system $X$, with measurement datum
\begin{equation}
    \mathfrak R=
    (R,0,\{r_\lambda\colon R_\lambda\hookrightarrow R\}_{\lambda\in\Lambda},\leq) \ .
\end{equation}
Let
\begin{equation}
    a\colon A\to X(s)
\end{equation}
be an $A$-generalised state over $s$. We say that $a$ has Lyapunov decay
toward $I$ if, for every $\lambda\in\Lambda$, the following property holds
eventually for morphisms $\alpha\colon s\to t$ in $\mathcal S$: the composite
\begin{equation}
    A\xrightarrow{a}X(s)\xrightarrow{X(\alpha)}X(t)
    \xrightarrow{L_t}R
\end{equation}
factors through
\begin{equation}
    r_\lambda\colon R_\lambda\hookrightarrow R \ .
\end{equation}
\end{definition}

\begin{theorem}[Lyapunov convergence criterion]\label{thm: Lyapunov criteria}
Let $X\colon \mathcal S\to\mathcal C$ be a dynamical system, where $\mathcal S$ is equipped with an eventuality
filter $ \mathcal E$. Let $\iota\colon \Delta I\Rightarrow X$ be an invariant subsystem equipped with a neighbourhood filter $\mathcal N_X(I)$.
Let $\mathfrak R=
    (R,0,\{r_\lambda\colon R_\lambda\hookrightarrow R\}_{\lambda\in\Lambda},\leq)$ be a Lyapunov measurement datum, and let $L=\{L_s\colon X(s)\to R\}_{s\in\mathcal S}$ be a categorical Lyapunov function for $I$ with values in $\mathfrak R$.
    
    If an $A$-generalised state
\begin{equation}
    a\colon A\to X(s)
\end{equation}
has Lyapunov decay toward $I$, then $a$ converges to $I$.
\end{theorem}

\begin{proof}
Let
\begin{equation}
    U\in\mathcal N_X(I)
\end{equation}
be an arbitrary neighbourhood. Since the sublevel neighbourhoods of $L$ form
a basis for $\mathcal N_X(I)$, there exists $\lambda\in\Lambda$ such that
\begin{equation}
    U^\lambda\leq U \ .
\end{equation}
By Lyapunov decay, there exists an eventual cosieve
\begin{equation}
    R_U\in\mathcal E_s
\end{equation}
such that, for every morphism $\alpha\colon s\to t$ in $R_U$, the composite
\begin{equation}
    A\xrightarrow{a}X(s)\xrightarrow{X(\alpha)}X(t)
    \xrightarrow{L_t}R
\end{equation}
factors through $R_\lambda$. By the universal property of pullback,
\begin{equation}
    X(\alpha)\circ a\colon A\to X(t)
\end{equation}
factors through
\begin{equation}
    U^\lambda_t\hookrightarrow X(t) \ .
\end{equation}
Since $U^\lambda\leq U$, it follows that $X(\alpha)\circ a$ factors through
$U_t\hookrightarrow X(t)$ for every $\alpha\in R_U$. This is precisely the
definition of convergence of $a$ to $I$.
\end{proof}

\begin{example}[Autonomous discrete-time systems]\label{ex: Lyapunov discrete autonomous}
Consider an autonomous discrete-time system
\begin{equation}
    X\colon B\mathbb N\to\mathsf{Top}
\end{equation}
defined by
\begin{equation}
    X(*)=M \ ,
    \qquad
    X(1)=f \ .
\end{equation}
Let $x_\infty\in M$ be an equilibrium, so that $f(x_\infty)=x_\infty$. It determines an invariant state
\begin{equation}
    \iota\colon \Delta 1\Longrightarrow X
\end{equation}
whose unique component is the point
\begin{equation}
    \iota_*\colon \{*\}\to M,
    \qquad
    \iota_*(*)=x_\infty \ .
\end{equation}

We use the standard Lyapunov measurement datum in $\mathsf{Top}$. Namely,
take $\mathbb R_{\geq 0}$ with its usual topology,
\begin{equation}
    R=\mathbb R_{\geq 0} \ ,
    \qquad
    0\colon \{*\}\to \mathbb R_{\geq 0}
\end{equation}
to be the point $0$, and for every $\varepsilon>0$ consider the canonical inclusion
\begin{equation}
    r_\varepsilon\colon R_\varepsilon=[0,\varepsilon)
    \hookrightarrow
    \mathbb R_{\geq 0} \ .
\end{equation}
The intervals $R_\varepsilon$ form a neighbourhood basis of $0$ in
$\mathbb R_{\geq 0}$. The preorder on maps to $R$ is the pointwise order: for continuous maps
\begin{equation}
    h,k\colon Y\to \mathbb R_{\geq 0} \ ,
\end{equation}
one writes $h\leq k$ if and only if $h(y)\leq k(y)$ for every $y\in Y$. This preorder is preserved by precomposition. Moreover, each $R_\varepsilon$ contains $0$ and is downward closed. Thus these data form
a Lyapunov measurement datum.

Now let
\begin{equation}
    L\colon M\to\mathbb R_{\geq 0}
\end{equation}
be a classical Lyapunov function satisfying
\begin{equation}
    L(x_\infty)=0  \qquad \text{and} \qquad  L(f(x))\leq L(x)
\end{equation}
for all $x\in M$. In the categorical notation, this is a family with a single
component
\begin{equation}
    L_*\colon X(*)\to R \ .
\end{equation}
Its sublevel stagewise neighbourhoods are
\begin{equation}
    U^\varepsilon_*
    =
    M\times_{\mathbb R_{\geq 0}}[0,\varepsilon)
    =
    \{\,x\in M\mid L(x)<\varepsilon\,\} \ .
\end{equation}
Since the intervals $[0,\varepsilon)$ form a neighbourhood basis of $0$ in
$\mathbb R_{\geq 0}$, the sublevel sets $U^\varepsilon_*$ form a neighbourhood basis of $x_\infty$ induced by the measurement $L$. The Lyapunov inequality implies
\begin{equation}
    f(U^\varepsilon_*)\subseteq U^\varepsilon_* \ ,
\end{equation}
and hence
\begin{equation}
    f^n(U^\varepsilon_*)\subseteq U^\varepsilon_*
\end{equation}
for every $n\in\mathbb N$. Thus the sublevel neighbourhoods are
forward-invariant. Therefore they form a Lyapunov basis for the neighbourhood
filter generated by the sublevel sets of $L$. With respect to this
neighbourhood filter, $L$ is a categorical Lyapunov function for $\iota$, and
the categorical Lyapunov stability principle recovers the usual sublevel-set
form of Lyapunov stability.

For convergence, equip $B\mathbb N$ with the eventuality filter generated by
\begin{equation}
    \mathcal E_N
    =
    \{\,n\colon *\to *\mid n\geq N\,\} \ ,
    \qquad N\in\mathbb N \ .
\end{equation}
Thus a property holds eventually precisely when it holds for all $n\geq N$ for some $N$. Now consider an initial state $x \in M$, viewed as a generalised state
\begin{equation}
    x\colon \{*\}\to M \ .
\end{equation}
The Lyapunov decay condition says that, for every $\varepsilon>0$, the
composite
\begin{equation}
    \{*\}\xrightarrow{x}M\xrightarrow{f^n}M\xrightarrow{L}
    \mathbb R_{\geq 0}
\end{equation}
eventually factors through
\begin{equation}
    [0,\varepsilon)\hookrightarrow \mathbb R_{\geq 0} \ .
\end{equation}
In concrete terms, this means that for every $\varepsilon>0$ there exists
$N\in\mathbb N$ such that, for all $n\geq N$,
\begin{equation}
    L(f^n(x))<\varepsilon \ .
\end{equation}

Thus the categorical Lyapunov decay condition recovers the classical
Lyapunov condition for convergence: the Lyapunov value along the trajectory
tends to zero, hence the trajectory eventually enters every sublevel
neighbourhood of the equilibrium.
\end{example}

\begin{example}[Guarded hybrid systems]\label{ex: lyapunov guarded hybrid}
Consider a guarded hybrid system as in Example~\ref{ex: hybrid guarded}. For
each mode $q\in Q$ and time $t\in\mathbb R$, let $X(q,t)$ be the
corresponding state space. Suppose that continuous evolution in mode $q$ from
time $s$ to time $t$ is given by a map
\begin{equation}
    \Phi^q_{s,t}\colon X(q,s)\to X(q,t) \ ,
    \qquad s\leq t \ ,
\end{equation}
and that an admissible jump
\begin{equation}
    e\colon q\to q'
\end{equation}
at time $t$ is specified by a guard and reset map
\begin{equation}
    g^e_t\colon G^e_t\hookrightarrow X(q,t),
    \qquad
    R^e_t\colon G^e_t\to X(q',t) \ .
\end{equation}

Let $ \iota\colon \Delta I\Rightarrow X$ be an invariant subsystem. Classically, a Lyapunov function is a family of continuous functions
\begin{equation}
    L_{q,t}\colon X(q,t)\to\mathbb R_{\geq 0},
    \qquad q\in Q,\ t\in\mathbb R \ ,
\end{equation}
satisfying the following conditions. First, $L$ vanishes on the invariant subsystem: for every stage $(q,t)$,
\begin{equation}
    L_{q,t}\circ \iota_{q,t}
    = 0 \ .
\end{equation}
Second, $L$ is non-increasing along both kinds of hybrid evolution. Along
continuous evolution, one requires
\begin{equation}
    L_{q,t}\bigl(\Phi^q_{s,t}(x)\bigr)
    \leq
    L_{q,s}(x)
\end{equation}
for all $x\in X(q,s)$ and all $s\leq t$. Along a guarded jump
$e\colon q\to q'$, one requires
\begin{equation}
    L_{q',t}(R^e_t(y))
    \leq
    L_{q,t}(g^e_t(y))
\end{equation}
for all $y\in G^e_t$.

We now view these data categorically. We use the same Lyapunov measurement datum as in Example~\ref{ex: Lyapunov discrete autonomous}. Then the family $L_{q,t}$ is precisely a family of components of a categorical Lyapunov function with values in this measurement datum.
The vanishing condition above is the categorical equation
\begin{equation}
    L_{q,t}\circ \iota_{q,t}
    =
    0\circ {!}_I \ ,
\end{equation}
and the two inequalities are exactly the non-increasing condition along the
generating morphisms.

For $\varepsilon>0$, the sublevel stagewise neighbourhoods are
\begin{equation}
    U^\varepsilon_{q,t}
    =
    X(q,t)\times_{\mathbb R_{\geq 0}}[0,\varepsilon)
    =
    \{\,x\in X(q,t)\mid L_{q,t}(x)<\varepsilon\,\} \ .
\end{equation}
The continuous-evolution inequality gives
\begin{equation}
    \Phi^q_{s,t}(U^\varepsilon_{q,s})
    \subseteq
    U^\varepsilon_{q,t} \ .
\end{equation}
The guarded-jump inequality gives
\begin{equation}
    R^e_t
    \bigl(
        (g^e_t)^{-1}(U^\varepsilon_{q,t})
    \bigr)
    \subseteq
    U^\varepsilon_{q',t} \ .
\end{equation}
Equivalently, using the guard as a subspace of $X(q,t)$, this says
\begin{equation}
    R^e_t
    \bigl(
        G^e_t\cap U^\varepsilon_{q,t}
    \bigr)
    \subseteq
    U^\varepsilon_{q',t} \ .
\end{equation}
Thus the sublevel neighbourhoods are preserved by both continuous evolutions
and admissible guarded jumps. Hence they form forward-invariant stagewise
neighbourhoods in the categorical sense.

Let $\mathcal N_X(I)$ be the stagewise neighbourhood filter generated by the
sublevel neighbourhoods
\begin{equation}
    U^\varepsilon=
    \bigl(U^\varepsilon_{q,t}\hookrightarrow X(q,t)\bigr)_{q,t} \ ,
    \qquad
    \varepsilon>0 \ .
\end{equation}
Then the family $\{U^\varepsilon\}_{\varepsilon>0}$ is, by construction, a Lyapunov basis for $\mathcal N_X(I)$. Therefore the
categorical Lyapunov stability principle gives stability of $I$ under all
admissible hybrid evolutions.

For convergence, consider an initial state $x_0$ at stage $X(q_0,t_0)$. We equip the source category with the eventuality filter at
$(q_0,t_0)$ generated by
\begin{equation}
    \mathcal E_{q_0,t_0}
    =
    \{\,\alpha\colon (q_0,t_0)\to(q,t)
        \mid t\geq T\,\} \ ,
    \qquad T\geq t_0 \ .
\end{equation}
A property of admissible evolutions out of $(q_0,t_0)$ holds
eventually if it holds after sufficiently large terminal time. For any admissible morphism
\begin{equation}
    \alpha_t \colon (q_0,t_0)\to (q(t),t)
\end{equation}
in the source category, we write
\begin{equation}
    x(t)
    \coloneqq
    X(\alpha_t)(x_0)
    \in
    X\big(q(t),t\big) \ .
\end{equation}
The Lyapunov decay condition for the initial state $x_0$ says that, for
every $\varepsilon>0$, one eventually has
\begin{equation}
    L_{q(t),t}\big(x(t)\big)<\varepsilon \ .
\end{equation}
Equivalently,
\begin{equation}
    L_{q(t),t}\big(x(t)\big)\xrightarrow{t \to \infty} 0 \ .
\end{equation}
Thus the categorical Lyapunov convergence
criterion recovers the classical Lyapunov condition for convergence of $x_0$ to the invariant subsystem $I$.
\end{example}

\end{document}